\documentclass[reqno, a4paper]{amsart}

\usepackage[english]{babel} 
\usepackage[hidelinks]{hyperref}
\usepackage{amsmath}
\usepackage{amsfonts}
\usepackage{amssymb}
\usepackage{amsthm}
\usepackage{tikz-cd}
\usepackage{mathrsfs}
\usepackage{mathtools}
\usepackage{enumerate}
\usepackage{url}
\usepackage{bbm}
\usepackage{tabularx}
\usepackage{comment}

\usepackage[all]{xy}

\newcommand{\R}{\mathbb{R}}
\newcommand{\Q}{\mathbb{Q}}
\newcommand{\Z}{\mathbb{Z}}

\newcommand{\id}[1]{{\rm Id}_{#1}} 
\newcommand{\Top}{\sf Top}
\newcommand{\TD}{\sf TD}

\newcommand{\MV}{{\sf MV}}

\newcommand{\BA}{{\sf BA}}
\newcommand{\KHaus}{{\sf KHaus}}
\DeclareMathOperator{\Spec}{{\rm Spec}}
\DeclareMathOperator{\Cp}{{\rm Cp}}
\newcommand{\LF}{\mathsf{LF}}
\newcommand{\LSimp}{\mathsf{LS}}
\newcommand{\Speck}{\mathsf{Sp}}

\DeclareMathOperator{\PP}{{\rm P}}
\DeclareMathOperator{\RR}{{\rm R}}
\DeclareMathOperator{\I}{{\rm I}}
\DeclareMathOperator{\Max}{{\rm Max}}

\DeclareMathOperator{\V}{\mathbb{V}}
\DeclareMathOperator{\Su}{\mathbb{S}}

\DeclareMathOperator{\Rad}{Rad}

\DeclareMathOperator{\C}{C}
\newcommand{\h}{\mathfrak{h}}

\DeclareMathOperator{\Sa}{S} 

\renewcommand{\leq}{\leqslant}
\renewcommand{\geq}{\geqslant}
\newcommand{\seq}{\subseteq}

\newcommand{\twopl}[2]{\langle #1, #2\rangle}

\newcommand{\p}{\mathfrak{p}}
\newcommand{\m}{\mathfrak{m}}
\renewcommand{\2}{\mathbbm{2}}

\newtheorem{theorem}{Theorem}[section]
\newtheorem{lemma}[theorem]{Lemma}
\newtheorem{corollary}[theorem]{Corollary}
\newtheorem{prop}[theorem]{Proposition}

\theoremstyle{definition}
\newtheorem{definition}[theorem]{Definition}
\newtheorem{remark}[theorem]{Remark}

\newcommand{\Set}{\mathsf{Set}}

\newcommand{\FinSets}{\ensuremath{\mathsf{fSet}}}

\newcommand{\Stone}{\mathsf{Stone}}
\newcommand{\Ring}{\mathsf{Ring}}

\newcommand{\lA}{\ell\mathsf{A}}

\newcommand{\opCat}[1]{\ensuremath{#1^\mathrm{op}}}

\newcommand{\calA}{\ensuremath{\mathcal A}}

\newcommand{\calC}{\ensuremath{\mathcal C}}

\newcommand{\calE}{\mathcal{E}}

\newcommand{\calS}{\ensuremath{\mathcal S}}

\newcommand{\pr}{\mathrm{pr}}  

\DeclareMathOperator{\Dec}{Dec}

\theoremstyle{remark}

\title{Separable MV-algebras and lattice-groups}
\author[V. Marra]{Vincenzo Marra}
\author[M. Menni]{Mat\'ias Menni}
\address[V. Marra]{Dipartimento di Matematica {\sl Federigo Enriques}, Universit\`a degli Studi di Milano, via Cesare Saldini 50, 20133 Milano, Italy.}
\address[M. Menni]{Conicet and Departamento de Matem\'atica, Universidad Nacional de La Plata, Calles 50 y 115, 1900 - La Plata, Argentina.}
\email[V. Marra]{vincenzo.marra@unimi.it}
\email[M. Menni]{matias.menni@gmail.com}

\thanks{2020 {\it Mathematics Subject Classification.
}
Primary: 06D35. Secondary: 06F20, 18B50, 12F10}

\keywords{MV-algebra, lattice-ordered Abelian group, strong order unit, separable algebra, extensive category, decidable object}

\begin{document}

\begin{abstract} General theory determines the notion of separable MV-algebra (equivalently, of separable unital lattice-ordered Abelian group). We establish the following structure theorem: An MV-algebra is separable if, and only if, it is a finite product of algebras of rational numbers---i.e., of subalgebras of the MV-algebra $[0,1]\cap\Q$. Beyond its intrinsic algebraic interest, this research is motivated by the long-term programme of developing the algebraic geometry of the opposite of the category of MV-algebras, in analogy with the classical case of commutative $K$-algebras over a field $K$.
\end{abstract}

\maketitle


\section{Introduction}\label{s:introduction}

For any field $K$, a (commutative) $K$-algebra is separable if, and only if, it is a finite product of finite separable field extensions of $K$. See, for example,  \cite[Corollary~{4.5.8}]{Ford2017}. The aim of the present paper is to establish the analogue of this fact for MV-algebras and lattice-groups. We show as our main result that an MV-algebra is separable exactly when it is a finite product of algebras of rational numbers---the subalgebras of $[0,1]\cap\Q$ (Theorem \ref{ThmMain}). By a well-known theorem of Mundici \cite{Mundici86}, the category of MV-algebras is equivalent to the category of lattice-ordered Abelian groups with a unit. We  frame our treatment in the language of MV-algebras, and postpone to the final Appendix \ref{a:l-groups} a synopsis of its  translation to lattice-groups.

While the main result of this paper holds independent algebraic interest, it finds  its deeper motivation in a broader mathematical landscape on which we offer some  comments in this introduction. 

As explained in \cite{Lawvere08}, some of Grothendieck’s algebro-geometric constructions may be abstracted to the context of extensive categories \cite{Lawvere91,CLW}. 
A category $\calC$ with finite coproducts is {\em extensive} if   the canonical functor 
\[{\calC/X \times \calC/Y \rightarrow \calC/(X + Y)}\]
is an equivalence  for every pair of objects $X$, $Y$ in $\calC$.  
Extensivity attempts to make explicit a most basic property of (finite) coproducts in categories `of spaces'. For instance, the category $\Top$ of topological spaces and continuous functions between them is extensive; the category of groups is not.

Extensive experience indeed confirms that conceiving an extensive category as a category `of spaces'  is a useful conceptual guide.  Essential to the development of Algebraic Geometry is the fact that $\opCat{\Ring}$, the opposite of the category of (commutative unital) rings, is extensive. 
(It easily follows that, for any ring $R$, the opposite of the category ${R/\Ring}$ of $R$-algebras is extensive.)  
Extensivity naturally determines a notion of {\em complemented} subobject. 
So, in an extensive category with finite products, it is also natural to consider the objects with complemented diagonal. These  are traditionally called {\em decidable objects}, and it is useful to think of them as the `discrete spaces' inside the category `of spaces' where they live. For instance, a topological space is decidable if, and only if, it is discrete. For any ring $R$, and any $R$-algebra $A$, let ${\Spec A}$ be the corresponding object in the extensive category ${\opCat{(R/\Ring)}}$.  Then ${\Spec A}$ is decidable if, and only if, $A$ is separable as an $R$-algebra. In other words, the separable $R$-algebras are precisely those for which the associated affine scheme is  decidable. 

Let us say that a category is {\em coextensive} if its opposite is extensive. In light of the above comments, an object in a coextensive category $\calA$ is called {\em separable} if the corresponding object in $\opCat{\calA}$ is decidable. 

The category $\MV$ of MV-algebras is coextensive.
This  provides the notion of separable MV-algebra that is the topic of the present paper. Explicitly, the MV-algebra $A$ is separable if, and only if, there is a homomorphism ${f \colon A + A \rightarrow A}$ such that the span
\[
\xymatrix{A & \ar[l]_-{\nabla} A + A \ar[r]^-{f} & A}
\]
is a product diagram, where $\nabla \colon A+A\to A$ denotes the codiagonal map.

The  geometry of $\opCat{\MV}$ has long been the subject of intensive hands-on study because of its striking connections with several areas of classical mathematics, from piecewise-linear topology to the geometry of numbers. 
The characterisation of decidable objects in $\opCat{\MV}$  that we present here was motivated by our ongoing long-term project to study of the `gros Zariski' topos determined by the theory of MV-algebras as the domain of a pre-cohesive geometric morphism \cite{Lawvere07}. We postpone the topos-theoretic consequences of separability to further publications; no Topos Theory is  required for the proof of the purely algebraic results in the present paper.

The plan of the paper is as follows. In Sections \ref{s:ext}, \ref{s:fcp-functors}, and \ref{s:dec} we introduce the necessary material to prove a sufficient condition for an extensive category with finite products to have the property that every decidable object is a finite coproduct of connected subterminals.
 In Section~\ref{s:MVcoext} we verify that $\MV$ is coextensive.
 In Theorem~\ref{t:superseparableNew} we characterise the  subterminal objects of $\opCat{\MV}$ as, in $\MV$,    the subalgebras of $[0,1]\cap\Q$.
In order to extend  Theorem \ref{t:superseparableNew} to a characterisation of separable MV-algebras we need to introduce the Pierce functor for $\MV$, an analogue of the standard ring-theoretic functor by the same name.  
The key fact is that the Pierce functor  preserves coproducts. To prove it, in Section \ref{s:stonereflection} we develop the required material on the connected-component functor $\pi_0$ in $\Top$.  Using  the theory of spectra of MV-algebras recalled in Section \ref{s:Spectra} along with the topological $\pi_0$ functor, we are able to show in  Theorem~\ref{t:Pierce} that the Pierce functor does preserve all coproducts. Theorems \ref{t:superseparableNew} and \ref{t:Pierce}  are combined in Section \ref{s:main} to obtain our main result, the mentioned characterisation of separable MV-algebras. We conclude Section \ref{s:main} with a discussion that points to further research aimed at enriching the connected-component functor on $\opCat{\MV}$ to an `arithmetic connected-component functor'; this functor, we submit, arises out of locally finite MV-algebras. Finally, in Appendix~\ref{a:l-groups} we collect the translation of our main results to lattice-groups.

\section{Extensive categories and connected objects}\label{s:ext}
In this section we recall the definition of extensive category and of connected object.
 For more details about extensive categories see, for example, \cite{Lawvere91,CLW} and references therein.

A category $\calC$ with finite coproducts  is called {\em extensive} if for every $X$ and $Y$ in $\calC$ the canonical functor ${\calC/X \times \calC/Y \rightarrow  \calC/(X + Y)}$
is an equivalence.
Examples of extensive categories are $\Set$ (sets and functions), $\FinSets$ (finite sets and functions), any topos,  $\Top$, $\KHaus$ (compact Hausdorff spaces and continuous maps), $\Stone$ (Stone\footnote{By a {\em Stone space} we mean a compact Hausdorff zero-dimensional space. Such spaces are often called {\em Boolean} in the literature.} spaces and continuous maps). The categories of rings, of Boolean algebras and of distributive lattices\footnote{Throughout the paper, with the exception of Appendix \ref{a:l-groups}, we assume distributive lattices to have top and bottom elements preserved by homomorphisms.} are coextensive.
See \cite{Lawvere08} and \cite{CMZ} for further examples.

In extensive categories coproduct injections are regular monomorphisms, 
coproducts of monomorphisms are monomorphisms, and
the initial object  is {\em strict} in the sense that  any map ${X \rightarrow 0}$ is an isomorphism. Also, extensive categories are closed under slicing.

\begin{definition}\label{DefDisjointCopros}
A coproduct ${in_0 \colon X \rightarrow X + Y \leftarrow  Y :{}in_1}$ is
\begin{enumerate}
\item  {\em disjoint} if the coproduct injections are monic and the commutative square 
$$\xymatrix{
0 \ar[d] \ar[r] & Y \ar[d]^-{in_1} \\
X \ar[r]_-{in_0} & X + Y
}$$
is a pullback; 
\item {\em universal} if for every arrow $Z \rightarrow X + Y$  the two pullback squares below exist
\[\xymatrix{
V \ar[d] \ar[r]& Z \ar[d] & \ar[l]W\ar[d] \\
X \ar[r]_-{in_0} & X + Y & \ar[l]^-{in_1} Y
}\]
and the top cospan is a coproduct diagram.
\end{enumerate}
\end{definition}

The following result is essentially    \cite[Proposition~{2.14}]{CLW}.

\begin{prop}\label{PropCharExtensivity}   A category with finite coproducts is extensive if, and only if, 
 coproducts are universal and disjoint.
\end{prop}

Assume from now on that $\calC$ is an extensive category.

A monomorphism ${u \colon U \rightarrow X}$ in $\calC$ is called {\em complemented} if there is a ${v \colon V \rightarrow X}$ such that the cospan
${u \colon U \rightarrow X \leftarrow V :{}v}$ is a coproduct diagram. In this case,  $v$ is the {\em complement} of $u$. Notice that complemented monomorphisms are regular monomorphisms  because they are coproduct injections.  
In the next definition, and throughout,  we  identify monomorphisms and subobjects whenever convenient.

\begin{definition}\label{DefConnected}
An object $X$ in $\calC$ is {\em connected} if it has exactly two complemented subobjects. 
\end{definition}

 In $\KHaus$ or $\Top$, an object is connected if and only if it has exactly two clopens. 
An object  $A$ in  ${\Ring}$ is connected as an object in $\opCat{\Ring}$ if and only if $A$ has exactly two idempotents.
We remark that, in general, connected objects are not closed under finite products.

\newcommand{\rmB}{\mathrm{B}}

For each $X$ in $\calC$ we let ${\rmB X}$ denote the poset of complemented subobjects of $X$.
We stress that  if ${u \colon U \rightarrow X}$ and ${v \colon V \rightarrow X}$ are two complemented monomorphisms  in $\calC$ and ${f \colon U \rightarrow V}$ is such that ${v f = u}$ then $f$ is complemented \cite[Lemma~{3.2}]{Gates98a}. So for any two complemented subobjects ${u, v}$ of $X$, there is no ambiguity  in writing ${u \leq v}$ since it means the same for $u$, $v$ considered as subobjects, or as complemented subobjects.

Extensivity easily  implies that the poset ${\rmB X}$ has finite infima, a bottom element, and an involution.
This structure  may be used to prove that ${\rmB X}$ is actually a Boolean algebra which interacts well with pullbacks in the sense that, for any map ${f \colon X \rightarrow Y}$ in $\calC$, pulling back along $f$ determines a Boolean algebra homomorphism ${\rmB Y \rightarrow \rmB X}$.
So, assuming that $\calC$ is well-powered, the assignment ${X \mapsto \rmB X}$ extends to a functor ${\calC \rightarrow \opCat{\BA}}$ between extensive categories that preserves finite coproducts.

We will use the following simple equivalences.

\begin{lemma}\label{LemDefIndec}  For any  object $X$ in $\calC$ the following are equivalent.
\begin{enumerate}
\item $X$ is connected.
\item $X$ is not initial and, for every complemented subobject ${u \colon U \rightarrow X}$, $U$ is initial or $u$ is an isomorphism.
\item $X$ is not initial and, for every coproduct diagram ${U \rightarrow X \leftarrow V}$, $U$ is initial or $V$ is initial.
\end{enumerate}
\end{lemma} 

\section{Finite-coproduct preserving  functors}\label{s:fcp-functors}
\label{SecAux}

Let $\calC$ and $\calS$ be extensive categories, and let ${L \colon \calC \rightarrow \calS}$ preserve finite coproducts. Such a functor preserves complemented monomorphisms so, for any $X$ in $\calC$,  $L$ induces a function ${\rmB X \rightarrow \rmB(L X)}$ which is actually a map in ${\BA}$, natural in $X$. (It is relevant to remark such a functor also preserves pullbacks along coproduct injections. See \cite[3.8]{Gates98a}.)

We will say that $L$ is {\em injective \textup{(}surjective/bijective\textup{)} on complemented subobjects} if and only if ${\rmB X \rightarrow \rmB(L X)}$ has the corresponding property for every $X$ in $\calC$.

\begin{lemma}\label{LemCharReflectionOfZero} The functor ${L \colon \calC \rightarrow \calS}$  is injective on complemented subobjects  if and only if it reflects $0$. In this case, $L$ also  reflects connected objects.
\end{lemma}
\begin{proof}
Assume first that $L$  is injective on complemented subobjects and let $X$ in $\calC$ be such that ${L X = 0}$.
Then ${\rmB(L X)}$ is the terminal Boolean algebra and, as ${\rmB X \rightarrow \rmB(L X)}$ is injective by hypothesis, ${\rmB X}$ is also trivial.
For the converse notice that if  $L$ reflects $0$ then the map ${\rmB X \rightarrow \rmB(L X)}$ in $\BA$ has trivial kernel for every $X$ in $\calC$.

To prove the second part of the statement assume that $X$ in $\calC$ is such that ${L X}$ is connected in $\calS$. 
If $X$ were initial then so would  ${L X}$  because $L$ preserves finite coproducts and, in particular, the initial object. So $X$ is not initial. 
Now assume that ${U \rightarrow X \leftarrow V}$ is a coproduct diagram.
Then so is ${L U \rightarrow L X \leftarrow L V}$. Since ${L X}$ is connected, either ${L U}$ or ${L V}$ is initial by Lemma~\ref{LemDefIndec}.
As $L$ reflects $0$, either $U$ or $V$ is initial, so $X$ is connected by the same lemma. (Alternatively, if ${\rmB X \rightarrow \rmB(L X)}$ is injective and its codomain is the initial Boolean algebra then so is the domain.)
\end{proof}

We will be particularly interested in extensive categories wherein  every object is a finite coproduct of connected objects. 
For example, $\FinSets$  satisfies this property, but neither  $\Set$ nor $\Stone$ does. 
If $\calA$ is the category of finitely presentable $K$-algebras for a field $K$, then ${\opCat{\calA}}$ also satisfies this property.

\begin{prop}\label{PropBijectiveOnSummands} If ${L \colon \calC \rightarrow \calS}$  is  bijective on complemented subobjects then the following hold.
\begin{enumerate}
\item The functor $L$ preserves connected objects.
\item For any object $X$ in $\calC$, if ${L X}$ is a finite coproduct of connected objects then so is $X$.
\item If every object in $\calS$  is a finite coproduct of connected objects then so is the case in $\calC$.
\item Assume that $\calC$ and $\calS$ have finite products and that $L$ preserves them. If $\calS$ is such that finite products of connected objects are connected then so is the case in $\calC$.
\end{enumerate}
\end{prop}
\begin{proof}
To prove the first item just notice that, by hypothesis, ${\rmB X \rightarrow \rmB(L X)}$ is an isomorphism for each $X$ in $\calC$. Hence if $X$ has exactly two complemented subobjects then so does ${L X}$.

Before proving the second item we establish an auxiliary fact. Let $X$ be  in $\calC$ and let ${u \colon U \rightarrow L X}$ be a complemented subobject in $\calS$ with connected $U$.
Then, as $L$ is surjective on complemented objects by hypothesis, there exists a  complemented subobject ${v \colon V \rightarrow X}$ in $\calC$ such that ${L v = u}$ as subobjects of ${L X}$.  Then ${L V \cong U}$ is connected,  so $V$ is connected by Lemma~\ref{LemCharReflectionOfZero}.
Thus, we have lifted the `connected component' $u$ of ${L X}$ to one of $X$.

To prove the second item let ${(u_i \mid i \in I)}$ be a  finite  family  of pairwise-disjoint complemented subobjects of ${L X}$ with connected domain whose join is the whole of ${L X}$.
For each ${i\in I}$,  let ${v_i}$ be the complemented subobject of  $X$ induced by ${u_i}$ as in the previous paragraph.
As $L$ reflects $0$, the family ${(v_i \mid i\in I)}$ is pairwise disjoint.
Also, ${L \bigvee_{i\in I} v_i = \bigvee_{i \in I} L v_i = \bigvee_{i\in I} u_i}$ is the whole of $LX$.
As $L$ is injective  on complemented subobjects, ${\bigvee_{i\in I} v_i}$ must be the whole of $X$.
In summary, we have lifted the finite coproduct decomposition of $L X$ to one of $X$.

The third item  follows at once from the second.

For the fourth item, let $X$ be the product of a finite family ${(X_i \mid i \in I)}$ of connected objects in $\calC$.
Then ${L X}$ is the product of  ${(L X_i \mid i \in I)}$ because $L$ preserves finite products.
Each ${L X_i}$ is connected because $L$ preserves connected objects by the first item, so ${L X}$ is connected by our hypothesis on $\calS$. 
Hence $X$ is connected by Lemma~\ref{LemCharReflectionOfZero}.
\end{proof}

We next prove a sufficient condition for a functor $L$ as above to be bijective on complemented subobjects.

\begin{lemma}\label{LemReflectionOfZeroAndPreservationsOfIndecsNewNew}
If ${L \colon \calC \rightarrow \calS}$ has a finite-coproduct preserving right adjoint, then   $L$ is bijective on complemented subobjects.
\end{lemma}
\begin{proof}
Let $R$ be the right adjoint to $L$ and let $\sigma$ and $\tau$ be the unit and counit of ${L \dashv R}$.
We show that $L$ is both injective and surjective on complemented subobjects.

To prove injectivity it is enough to show that $L$ reflects $0$ (Lemma~\ref{LemCharReflectionOfZero}).
So let $X$ be an object in $\calC$ such that ${L X}$ is initial.
Then we may transpose the isomorphism ${L X \rightarrow 0}$ in $\calS$ to a map ${X \rightarrow R 0}$, but ${R 0 = 0}$ because $R$ is assumed to preserve finite coproducts. 
Since the initial object is strict,  $X$ is initial.

We next show that $L$ is surjective on complemented subobjects.
Let ${u \colon U \rightarrow L X}$ be a complemented monomorphism.
Then ${R u}$ is complemented so the left pullback square below exists
\[\xymatrix{
V \ar[d]_-v \ar[r]   & R U \ar[d]^-{R u} && L V \ar[d]_-{L v} \ar[r]  & L(R U) \ar[d]^-{L(R u)} \ar[r]^-{\tau} &  U \ar[d]^-{u} \\
X \ar[r]_-{\sigma} & R (L X)                    && L X \ar[r]_-{L\sigma}    & L(R (L X))   \ar[r]_-{\tau} & L X
}\]
by extensivity of $\calC$. Then the  two squares on the right above obviously commute, and the bottom composite is the identity. Moreover, \cite[Lemma~{3.7}]{Gates98a} implies that both squares are pullbacks, so ${u}$ and ${L v}$ coincide as subobjects of $LX$. 
\end{proof}

Combining Lemma~\ref{LemReflectionOfZeroAndPreservationsOfIndecsNewNew} and Proposition~\ref{PropBijectiveOnSummands} we obtain the following.

\begin{corollary}\label{CorLiftingOfFiniteDecompositions} 
Assume that ${L \colon \calC \rightarrow \calS}$ has a finite-coproduct preserving right adjoint. If every object in $\calS$  is a finite coproduct of connected objects then so is the case in $\calC$.
\end{corollary}

\section{Decidable objects}\label{s:dec}

Let $\calC$ be an extensive category with finite products.
In particular, $\calC$ has a terminal object $1$.
An object $X$ is called {\em subterminal} if  the unique map ${X \rightarrow 1}$ is monic.

\begin{lemma}\label{LemCharSubterminals} 
For any object $X$ in $\calC$, the following are equivalent.
\begin{enumerate}
\item The object $X$ is subterminal.
\item The diagonal ${\Delta \colon X \rightarrow X\times X}$ is an isomorphism.
\item The projections ${\pr_0, \pr_1 \colon X\times X \rightarrow X}$ are equal.
\end{enumerate}
\end{lemma}
\begin{proof}
The first item implies the second because for any monomorphism  ${X \rightarrow 1}$ the following diagram 
$$\xymatrix{
X \ar[d]_-{id} \ar[r]^-{id} & X \ar[d]^-{!} \\
X \ar[r]_-{!} & 1
}$$
is a pullback. 
The second item implies the third because any map has at most one inverse.
To prove that the third item implies the first, let ${f, g \colon Y \rightarrow X}$. Then there exists a unique map ${\twopl{f}{g} \colon Y \rightarrow X \times X}$ such that ${\pr_0 \twopl{f}{g} = f}$ and ${\pr_1 \twopl{f}{g} = g}$.
So ${f = \pr_0 \twopl{f}{g} = \pr_1 \twopl{f}{g} = g}$.
That is, for any object $Y$ there is a unique map ${Y \rightarrow X}$.
This means that the unique map ${X \rightarrow 1}$ is monic.
\end{proof}

We stress that extensivity plays no r\^{o}le in Lemma~\ref{LemCharSubterminals}, which is a general fact about categories with finite products.

\begin{definition}\label{DefDecidable} An object $X$ in $\calC$ is {\em decidable} if the diagonal ${\Delta \colon X \rightarrow X \times X}$ is complemented.
\end{definition}

\begin{remark}\label{r:subt_are_dec}Lemma~\ref{LemCharSubterminals} shows that  subterminal objects in $\calC$ are decidable, and that they may be characterised as those decidable objects  $X$ such that the diagonal ${\Delta \colon X \rightarrow X \times X}$ not only is complemented, but is actually an isomorphism. 
\end{remark}

The full subcategory of decidable objects will be denoted by ${\Dec{\calC} \rightarrow \calC}$. 
If $\calC$ is \emph{lextensive} (i.e.\ extensive and with finite limits) it follows from \cite{CarboniJanelidze} that ${\Dec{\calC}}$ is lextensive and that the inclusion ${\Dec{\calC} \rightarrow \calC}$ preserves finite limits, finite coproducts and that it is closed under subobjects. Moreover, for any $X$, $Y$ in $\calC$, ${X + Y}$ is decidable if, and only if, both $X$ and $Y$ are decidable. 
On the other hand, arbitrary coproducts of decidable objects need not be decidable---consider, for instance, an infinite copower of the terminal object in $\KHaus$ or $\Stone$.

\begin{prop}\label{PropIndecomposablesubterminals} For any object $X$ in  $\calC$  the following are equivalent:
\begin{enumerate}
\item $X$ is subterminal and connected.
\item $X$ is decidable and ${X \times X}$ is connected.
\end{enumerate}
\end{prop}
\begin{proof}
If $X$ is subterminal and connected then ${\Delta \colon X \rightarrow X\times X}$ is an isomorphism by Lemma~\ref{LemCharSubterminals}. 
So $X$ is decidable and ${X\times X}$ is as connected as $X$.

For the converse assume that $X$ is decidable and that ${X \times X}$ is connected.
Decidability means that  the subobject ${\Delta \colon X \rightarrow X \times X}$ is complemented;  as  ${X \times X}$ is connected,  $X$ is initial or ${\Delta \colon X \rightarrow X \times X}$ is an isomorphism by Lemma~\ref{LemDefIndec}. But $X$ is not initial (because ${X\times X}$ is connected) so ${\Delta \colon X \rightarrow X \times X}$ is an isomorphism. Then $X$ is as connected as ${X\times X}$, and $X$  is subterminal by Lemma~\ref{LemCharSubterminals}.
\end{proof}

Let $\calS$ be another extensive category with finite products and let ${L \colon \calC \rightarrow \calS}$ preserve finite products and finite coproducts.

\begin{lemma}\label{LemConnectedImpliesIndecomposableNew} Assume that $L$ reflects $0$ and that 
 $1$ is connected in $\calS$.  Then the following hold for every $X$ in $\calC$.
\begin{enumerate}
\item If ${L X = 1}$ then $X$ is connected. 
\item If $X$ in $\calC$ is decidable and ${L X = 1}$ then $X$ is subterminal.
\end{enumerate}
\end{lemma}
\begin{proof}
The functor $L$ reflects $0$ so it  reflects connected  objects by Lemma~\ref{LemCharReflectionOfZero}.
As $1$ is connected in $\calS$ by hypothesis, ${L X = 1}$ implies $X$ connected.

If ${L X = 1}$ then ${L (X \times X) = L X \times L X = 1}$.
So ${X \times X}$  is connected by the first item.
Therefore $X$ is subterminal by Proposition~\ref{PropIndecomposablesubterminals}.
\end{proof}

It easily follows from the definition of decidable object that  $L$ preserves decidable objects. In more detail, the preservation properties of $L$ imply that  the left-bottom composite below 
\[\xymatrix{
\Dec{\calC} \ar[d] \ar@{.>}[r] & \Dec{\calS} \ar[d] \\
\calC \ar[r]_-{L} & \calS
}\]
factors uniquely through the right inclusion and, moreover,  ${\Dec{\calC} \rightarrow \Dec{\calS}}$ preserves  finite products and finite coproducts.
In fact,  ${\Dec{\calC} \rightarrow \Dec{\calS}}$ preserves all the finite limits that $L$ preserves (because  the subcategories of decidable objects are closed under finite limits).

Additionally assume from now on that ${L \colon \calC \rightarrow \calS}$ has a finite-coproduct preserving right adjoint ${R \colon \calS \rightarrow \calC}$.

Notice that under the present hypotheses both $L$ and $R$ preserve  finite products and finite coproducts. 
It follows that the adjunction ${L\dashv R}$  restricts to one between ${\Dec{\calS}}$ and ${\Dec{\calC}}$.

\begin{corollary}\label{CorMainNew} 
If  every decidable object in $\calS$ is a finite coproduct of connected objects then so is the case in $\calC$.
\end{corollary}
\begin{proof}
The adjunction ${L \dashv R \colon \calS \rightarrow \calC}$ restricts to one ${L' \dashv R' \colon \Dec{\calS} \rightarrow \Dec{\calC}}$,
and every object in ${\Dec{\calS}}$ is a finite coproduct of connected objects by hypothesis.
So we may apply Corollary~\ref{CorLiftingOfFiniteDecompositions} to ${L'\colon \Dec{\calC} \rightarrow \Dec{\calS}}$
\end{proof}

Because $\calS$ is lextensive, there exists an essentially unique coproduct preserving functor ${\FinSets \rightarrow \calS}$ that also preserves the terminal object.
The functor sends a finite set $I$ to the copower ${I\cdot 1}$ in $\calS$.
The categories $\FinSets$, $\Stone$, and other examples have the property that this functor  ${\FinSets \rightarrow \calS}$ coincides with ${\Dec{\calS} \rightarrow \calS}$. Notice that if this condition holds then $1$ is connected in $\calS$, because ${\FinSets = \Dec{\calS} \rightarrow \calS}$ is closed under subobjects and preserves $1$.

\begin{prop}\label{PropMain} If the canonical functor ${\FinSets \rightarrow \calS}$ coincides with ${\Dec{\calS} \rightarrow \calS}$ then every decidable object in $\calC$ is a finite coproduct of connected subterminals.
\end{prop}
\begin{proof}
By Corollary~\ref{CorMainNew} every decidable object in $\calC$ is a finite coproduct of connected objects. So it is enough to prove that every connected decidable  object in $\calC$ is subterminal. For this, let $X$ be connected and decidable. 
Then ${L X}$ is decidable, because $L$ preserves finite products and finite coproducts, and it is connected by Lemma~\ref{LemReflectionOfZeroAndPreservationsOfIndecsNewNew} and Proposition~\ref{PropBijectiveOnSummands}.
By hypothesis, the canonical ${\FinSets \rightarrow \calS}$ coincides with ${\Dec{\calS} \rightarrow \calS}$ so ${L X = 1}$. 
Hence $X$ is decidable and ${L X = 1}$. Therefore  $X$ is subterminal  by Lemma~\ref{LemConnectedImpliesIndecomposableNew}.
\end{proof}

For a lextensive category $\calE$ we have considered several conditions.
\begin{enumerate}
\item Every decidable object is a finite coproduct of connected objects.
\item Every decidable object is a finite coproduct of connected subterminals.
\item The canonical functor ${\FinSets \rightarrow \calE}$ coincides with the inclusion ${\Dec{\calE} \rightarrow \calE}$.
\end{enumerate}
\noindent For a field $K$, ${\opCat{(K/\Ring)}}$ satisfies the first condition but not the second.  The categories $\Stone$ and $\KHaus$ satisfy the third condition.
The  third condition implies the second which, in turn, implies the first.
Proposition~\ref{PropMain} shows that for certain adjunctions ${L \dashv R \colon \calS \rightarrow \calC}$, if $\calS$ satisfies the third condition then $\calC$ satisfies the second. This will be used to prove  that ${\opCat{\MV}}$ satisfies the second condition (Theorem~\ref{ThmMain}).

\section{The coextensive category of MV-algebras}\label{s:MVcoext}

 For background on MV-algebras we refer to the standard textbooks \cite{CignoliEtAlBook, MundiciAdvanced}, of which we also follow the notation. 
In this section we show that $\MV$ is coextensive by proving that products are {\em co}disjoint and {\em co}universal (Proposition~\ref{PropCharExtensivity}). 

\begin{lemma}\label{LemCodisjointProds} 
Let $\calA$ be a regular category with finite colimits.
If ${0 \rightarrow 1}$ is a regular epimorphism then  products are codisjoint.
\end{lemma}
\begin{proof}
Let $A$ be an object in $\calA$.
As the composite ${0 \rightarrow A \rightarrow 1}$ is a regular epimorphism by hypothesis, so is ${A \rightarrow 1}$ by regularity of $\calA$.
That is, not only ${0 \rightarrow 1}$ but actually any ${A \rightarrow 1}$ is a regular epimorphism. 
As every regular epimorphism is the coequalizer of its kernel pair, ${A \rightarrow 1}$ is the coequalizer of the two projections ${A \times A \rightarrow A}$.
Also, as products of regular epimorphisms are epimorphisms, the product of ${id \colon A \rightarrow A}$ and ${B \rightarrow 1}$ is a regular epimorphism ${A \times B \rightarrow A \times 1}$. That is, the projection ${A \times B \rightarrow A}$ is a regular epimorphism.

To complete the proof we recall a basic fact about colimits:
for a commutative diagram as on the left below
\[\xymatrix{
E \ar[d]_-e \ar[r]<+1ex>^-{e_0} \ar[r]<-1ex>_-{e_1} & D  \ar[d]_-d \ar[r] & B \ar[d] &
  (A\times A) \times B \ar[d]_-{\pr_0} \ar[rr]<+1ex>^-{\pr_0 \times B} \ar[rr]<-1ex>_-{\pr_1 \times B} && A\times B  \ar[d]_-{\pr_0} \ar[r]^-{\pr_1} & B \ar[d] \\ 
F \ar[r]<+1ex>^-{f_0} \ar[r]<-1ex>_-{f_1} & A \ar[r] & Q &
  A\times A  \ar[rr]<+1ex>^-{\pr_0} \ar[rr]<-1ex>_-{\pr_1} && A   \ar[r] & 1
}\]
such that ${d e_i = f_i e}$ for ${i \in \{0, 1\}}$,  the top  and bottom forks are coequalizers and $e$ is epic, the inner right square is a pushout.  Applying this observation to the diagram on the right above we obtain that the inner right square in that diagram is a pushout.
\end{proof}

In particular, if $\calA$ is the category of models for an algebraic theory with at least one constant then the initial object $0$ is non-empty and so ${0 \rightarrow 1}$ is a regular epimorphism. This is the case, of course, for ${\calA = \MV}$.

 In $\Ring$, couniversality of products is entailed by the intimate relationship between idempotents and product decompositions. The situation for $\MV$ is analogous. An element $b$ of an MV-algebra $A$ is called \emph{Boolean} if it satisfies one of the following equivalent conditions (see \cite[1.5.3]{CignoliEtAlBook}):

\[ 
b\oplus b=b \quad\quad
 b\odot b=b \quad\quad
 b\vee\neg b=1 \quad\quad
 b\wedge \neg b=0.
\]

For  ${x\in A}$ we let ${A \rightarrow A[x^{-1}]}$ be the quotient map induced by the congruence on $A$ generated by the pair $(x,1)$.

\begin{lemma}\label{LemPOofLocalizations} For any ${f \colon A \rightarrow B}$ in $\MV$ the following diagram is a pushout
\[\xymatrix{
A \ar[d]_-{f} \ar[r] & A[x^{-1}] \ar[d] \\
B \ar[r] & B[(f x)^{-1}]
}\]
where the right vertical map is the unique one making the square commute.
\end{lemma}
\begin{proof}
Standard, using the universal property of the (horizontal) quotient homomorphisms.
\end{proof}

\begin{lemma}\label{l:universalinv}
For any MV-algebra $A$ and every Boolean element $x\in A$, let $\langle \neg x \rangle$ be the ideal of $A$ generated by $\{\neg x\}$. Then the quotient $q\colon A\to A/\langle \neg x\rangle$ has the universal property of ${A \rightarrow A[x^{-1}]}$.
\end{lemma}
\begin{proof}
If $k\colon A \to B$ is such that ${k x = 1}$  then $\neg x \in \ker{k}$, so $\langle\neg x\rangle\seq\ker{k}$. By the universal property of quotients there is exactly one homomorphism $c\colon A/\langle\neg x\rangle\to  C$ such that $cq=k$.
\end{proof}

\begin{lemma}\label{l:productsplitting}
In  $\MV$, the diagram
\[\xymatrix{
D & \ar[l]^-{q_0} C \ar[r]_-{q_1} & E
}\]
is a product  precisely when there exists a Boolean element ${x\in C}$ such that $q_0$ has the universal property of ${C \rightarrow C[(\neg x)^{-1}]}$ and   $q_1$ has the universal property of ${C \rightarrow C[x^{-1}]}$. 
When this is the case, the element $x\in C$ with the foregoing property is unique.
\end{lemma}
\begin{proof}
Assume the diagram is a product. Then there is a unique  ${x\in C}$  such that $q_ix=i$, $i=0,1$. This $x$ is Boolean because $0$ and $1$ are. Hence $\neg x$ is Boolean too, and thus  $\oplus$-idempotent; therefore, $\langle  \neg x \rangle=\{c\in C\mid c \leq  \neg x\}$. If $c\leq \neg x$ then $q_1c\leq q_1(\neg x)=0$, so $q_1c=0$ and $c\in \ker{q_1}$. If $c\in \ker{q_1}$ then $q_1c=0\leq q_1(\neg x)$ and $q_0c\leq 1=q_0(\neg x)$, so $c\leq 	\neg x$ by the definition of product order. We conclude $\ker{q_1}=\langle \neg x \rangle$. The projection $q_1$ is surjective  so Lemma \ref{l:universalinv} entails that $q_1$ has the universal property of ${C \rightarrow C[x^{-1}]}$. 
An entirely similar argument applies to $q_0$.

Conversely, assume $q_0$ and $q_1$ have the universal properties in the statement. 
By Lemma \ref{l:universalinv} we may identify ${q_0}$ with ${C \rightarrow C/\langle x\rangle}$ and $q_1$ with ${C \rightarrow C/\langle \neg x\rangle}$. So it is enough to show that the canonical ${C \rightarrow C/\langle x\rangle \times  C/\langle \neg x\rangle}$ is bijective.
Injectivity follows because if $c\leq x, \neg x$ then $c\leq x\wedge \neg x=0$, so ${\langle x\rangle \cap \langle\neg x\rangle = 0}$.
To prove surjectivity, let ${(q_0 c_0 , q_1 c_1) \in C/\langle x\rangle \times  C/\langle \neg x\rangle}$ with ${c_0, c_1 \in C}$ and consider 
${c = (c_0 \wedge \neg x) \vee (c_1 \wedge x) \in  C}$. It is easy to check that ${C \rightarrow C/\langle x\rangle \times  C/\langle \neg x\rangle}$ sends $c$ in the domain  to ${(q_0 c_0 , q_1 c_1)}$ in the codomain.
\end{proof}
\begin{remark}\label{r:notnew}
The content of Lemma \ref{l:productsplitting} is far from new, cf.\ e.g.\ \cite[Section 6.4]{CignoliEtAlBook} and \cite[Proposition~{3.9}]{CMZ}. However, having expressed that content in the form that is most suitable for the sequel, we have included a proof for the reader's convenience.
\end{remark}

\begin{prop}\label{PropMVisCoextensive} 
$\MV$ is coextensive.
\end{prop}
\begin{proof}
Any algebraic category is complete and cocomplete, so in particular it has finite products and pushouts.
We appeal to the characterization of extensive categories in Proposition~\ref{PropCharExtensivity}.
Codisjointness of products follows from Lemma~\ref{LemCodisjointProds} or from a direct calculation observing that the projections of a product ${A \times B}$ send ${(0, 1)}$ to $0$ and $1$ respectively, so ${0 = 1}$ must hold in the pushout.

It remains to show that products are couniversal.
So we consider the pushout of a product diagram as below
\[\xymatrix{
A \ar[d]_-{h} &  \ar[l]_-{{\rm pr}_0}  A\times B \ar[d]^-{f} \ar[r]^-{{\rm pr}_1} & B \ar[d]^-{k} \\
D & \ar[l]^-{q_0} C \ar[r]_-{q_1} & E
}\]
and prove that the bottom span is product diagram.
Indeed, observe that the Boolean element ${(0, 1) \in A\times B}$ is sent to the Boolean element ${x\coloneqq {f(1, 0) \in C}}$ so, by Lemma~\ref{l:productsplitting},  it is enough to check that ${q_0}$ inverts ${\neg x}$ and $q_1$ inverts ${x}$;
but this follows from Lemma~\ref{LemPOofLocalizations}.
\end{proof}

Although it was not necessary to prove the main result of this section, it seems worthwhile to observe that, in the context of algebraic categories, Lemma~\ref{LemCodisjointProds} may be strengthened to a characterisation.

\begin{prop}\label{l:codisjprodalgcat}
In any algebraic category, binary products are codisjoint if, and only if, the initial algebra has non-empty underlying set.
\end{prop}
\begin{proof}
If the initial algebra $0$ is not empty then the unique map ${0 \rightarrow 1}$ is a regular epimorphism so we can apply 
Lemma~\ref{LemCodisjointProds}.
For the converse implication notice that the following square
\[\xymatrix{
0 \times 0 \ar[d] \ar[r] & 0 \ar[d]\\
0\ar[r] & 1
}\]
is a pushout by hypothesis. As any of the projections ${0\times 0 \rightarrow 0}$ is split epic, its pushout ${0 \rightarrow 1}$ is a regular epimorphism, so $0$ must be non-empty.
\end{proof}

\section{Subterminals in \opCat{\MV}, and rational algebras}\label{s:subt}
The aim of this section is to characterize subterminal objects in ${\opCat{\MV}}$. 
Perhaps unexpectedly, the following fact will play an important r\^{o}le.

\begin{lemma}\label{LemMonosAreStableUnderPOinMV} Monomorphisms in $\MV$ are stable under pushout. 
\end{lemma}
\begin{proof}
It is well known \cite{KissEtAl82} that, in algebraic categories,  stability of monomorphisms under pushout is equivalent to the conjunction of the Amalgamation Property (AP) and of the Congruence Extension Property (CEP).
Pierce proved the AP for Abelian lattice-groups in \cite{Pierce76}, and Mundici \cite[Proposition 1.1]{Mundici88} observed that Pierce's result transfers through the functor $\Gamma$ to MV-algebras. For a different proof of the AP for Abelian lattice-groups and MV-algebras, see \cite[Theorems 36 40]{MetcalfeEtAl14}. The CEP for MV-algebras was proved in \cite[Proposition 8.2]{MundiciGispert}; for an alternative proof, see \cite[p.\ 230]{MundiciAdvanced}. For yet another proof in the more general context of residuated lattices, see \cite[Corollary 44]{MetcalfeEtAl14}. 
\end{proof}
Most of the work will be done on the algebraic side,  so it is convenient to start with an arbitrary 
category  $\calA$  with finite coproducts whose  initial object is denoted   $0$.
As suggested above, we  concentrate on the objects $A$ such that the unique map ${0 \rightarrow A}$ is epic. Notice that such an object is exactly a subterminal object in $\opCat{\calA}$, but we prefer to avoid introducing new terminology such as `cosubterminal' or `supra-initial'.
For  convenience we state here the dual of Lemma~\ref{LemCharSubterminals}.

\begin{lemma}\label{LemCharSubterminalsOP} For any object $A$ in $\calA$, the following are equivalent:
\begin{enumerate}
\item The map ${0 \rightarrow A}$ is epic. 
\item The codiagonal ${\nabla \colon A + A \rightarrow A}$ is an isomorphism.
\item The coproduct injections ${in_0 , in_1 \colon A \rightarrow A + A}$ are equal.
\end{enumerate} 
\end{lemma}

We shall also need a simple  auxiliary fact.

\begin{lemma}\label{LemSpecialQuotientsOfSubterminalAreSubterminalOP}
Let ${0\rightarrow A}$ be epic and ${m \colon B \rightarrow A}$ be a map.
If the coproduct map ${m + m \colon B + B \rightarrow A + A}$ is monic then ${0 \rightarrow B}$ is epic.
\end{lemma}
\begin{proof}
The following square commutes
\[\xymatrix{
B + B \ar[d]_-{m + m} \ar[r]^-{\nabla} & B \ar[d]^-m \\
A + A \ar[r]_-{\nabla} & A
}\]
by naturality of the codiagonal. The bottom map is an isomorphism by Lemma~\ref{LemCharSubterminalsOP}, and the left vertical map is monic by hypothesis. So the top map is also monic, as well as  split epic.
\end{proof}

Assume from now on that $\calA$ has finite colimits and that monomorphisms are stable under pushout. We stress that this stability property is quite restrictive. For instance, it does not hold in $\Ring$. On the other hand, we already know that it holds  in $\MV$ by Lemma~\ref{LemMonosAreStableUnderPOinMV}.

\begin{lemma}\label{LemCharSubterminalsExtraOP} 
The map ${0 \rightarrow A}$ is epic
 if, and only if, for every monomorphism ${B \rightarrow A}$, ${0 \rightarrow B}$ is epic.
\end{lemma}
\begin{proof}
One direction is trivial and does not need stability of monomorphisms.
For the converse observe that, as monomorphisms are stable under pushout, finite coproducts of monomorphisms are monic.
So we can apply  Lemma~\ref{LemSpecialQuotientsOfSubterminalAreSubterminalOP}.
\end{proof}

The following is a further  auxiliary fact.

\begin{lemma}\label{LemWeirdMoreGeneralOPnew}
For any ${d\colon A \rightarrow D}$ and ${e\colon B \rightarrow A}$ in $\calA$, if $e$ is epic and the composite  ${d e \colon B \rightarrow D}$ is monic then $d$ is an monic.
\end{lemma}
\begin{proof}
The right square below is trivially a pushout and, since ${e \colon B \rightarrow A}$ is epic, the left square is also a pushout
$$\xymatrix{
B \ar[d]_-{e} \ar[r]^-{e} & A \ar[d]^-{id} \ar[r]^-{d} & D \ar[d]^-{id} \\
A  \ar[r]_-{id} & A \ar[r]_-{d} & D
}$$
so the rectangle is  a pushout too. As the top composite is monic, and these are are stable under pushout by hypothesis, the bottom map is monic. 
\end{proof}

We emphasise the next particular case of Lemma~\ref{LemWeirdMoreGeneralOPnew}.

\begin{lemma}\label{LemWeirdMoreParticularOP} 
Let ${d \colon A \rightarrow D}$ be a regular epimorphism in $\calA$.
If ${0 \rightarrow A}$ is epic and ${0 \rightarrow D}$ is monic then $d$ is an isomorphism.
\end{lemma}

Assume now that our category $\calA$ with finite colimits and stable monomorphisms  has a terminal object $1$ such that for any object $A$ in $\calA$ the unique ${A \rightarrow 1}$ is a regular epimorphism.
This is common in  algebraic categories.

A {\em quotient} of $A$ in $\calA$ is  an equivalence class of regular epimorphisms with domain $A$, where two such are equivalent if they are isomorphic as objects of ${A/\calA}$.

An object $A$ is {\em simple} if it has exactly two quotients, namely, those represented by ${A \rightarrow 1}$ and ${id\colon A \rightarrow A}$.
So, if $\calA$ is an algebraic category, then an object is simple  if and only if  it has exactly two congruences.

To motivate the hypotheses of the following lemma observe that for every object $A$ in $\BA$, $A$ is terminal or ${0 \rightarrow A}$ is monic.
  Similarly for $\MV$ and for ${K/\Ring}$ with $K$ a field. In contrast, that is not the case in ${\Ring}$.

\begin{lemma}\label{LemSubterImpliesCosimpleOP}  
If for every object $D$ of $\calA$,  $D$ is terminal or ${0 \rightarrow D}$ is monic,  then for every epic ${0 \rightarrow A}$   the following hold.
\begin{enumerate}
\item  $A$ is simple or terminal.
\item If ${m \colon B \rightarrow A}$ is monic then ${B + B}$ is simple or  terminal.
\end{enumerate}
\end{lemma}
\begin{proof}
To prove the first item let ${d \colon A \rightarrow D}$ be a regular epimorphism. Then $D$ is terminal  or ${0 \rightarrow D}$ is monic by hypothesis.
If ${0 \rightarrow D}$ is monic then $d$ is an isomorphism by Lemma~\ref{LemWeirdMoreParticularOP}.
So the only possible quotients of $A$ are ${A \rightarrow 1}$ or ${id \colon A \rightarrow A}$. So $A$ is terminal or simple.

To prove the second item first recall that epimorphisms are closed under coproduct.
Then recall that, as monomorphisms are stable by hypotheses, they are closed under finite coproducts.
Therefore, ${m + m \colon B + B \rightarrow A + A}$ is a monomorphism 
and ${0 = 0 + 0 \rightarrow A + A}$ is epic. 
So, by Lemma~\ref{LemCharSubterminalsExtraOP}, ${0\rightarrow B + B}$ is also epic. The first item implies that ${ B + B}$ is simple or terminal.
\end{proof}

The material in this section applies to the case ${\calA = \MV}$, so we may now prove our first MV-algebraic result. For the proof we require   a standard  fact from the theory of MV-algebras and lattice-groups, which will also find further application  later in this paper.
An ideal $\m$ of the MV-algebra $A$ is \emph{maximal} if it is proper, and inclusion-maximal amongst proper ideals of $A$; equivalently, the quotient $A/\m$ is a simple algebra.

\begin{lemma}[H\"older's Theorem {\cite{Holder01}} for MV-algebras {\cite[3.5.1]{CignoliEtAlBook}}]\label{MVHolder}
For every MV-algebra $A$, and for every maximal ideal $\m$ of $A$,  there is  exactly one  homomorphism of MV-algebras \[\h_{\m}\colon\tfrac{A}{\m}\longrightarrow [0,1],\]
and this homomorphism is injective.
\end{lemma}
In connection with the result that follows, let us explicitly recall that the initial object $0$ in $\MV$ is the two-element Boolean algebra $\{0,1\}$.
\begin{theorem}\label{t:superseparableNew} For any  MV-algebra $A$ the following are equivalent.
\begin{enumerate}[\textup{(}i\textup{)}]
\item\label{i:new3} $A$ is a subalgebra of $[0,1]\cap\Q$.
\item $A$ is non-trivial and the unique map ${0 \rightarrow A}$ is  epic.
\item\label{i:new2} The unique map ${0 \rightarrow A}$ is monic and epic.
\item $A$ is simple and ${0 \rightarrow A}$ is  epic.
\end{enumerate}
\end{theorem}
\begin{proof}
If ${A \subseteq [0,1]\cap\Q}$ then $A$ is certainly non-trivial,  and  \cite[Proposition~{7.2}]{MundiciAdvanced} shows that the coproduct inclusions  
${in_0, in_1 \colon A \rightarrow A + A}$ are equal. 
So  ${0 \rightarrow A}$ is epic by Lemma~\ref{LemCharSubterminalsOP}.

The second and third items are clearly equivalent, and they imply the fourth by  Lemma~\ref{LemSubterImpliesCosimpleOP}.

Finally, assume that $A$ is simple and that ${0 \rightarrow A}$ is epic. 
By H\"older's Theorem (Lemma \ref{MVHolder}) together with simplicity,  there is exactly one monomorphism ${A\to [0,1]}$.
Now let ${r \in A}$ and write ${\iota \colon A_r \rightarrow A}$ for the subalgebra of $A$ generated by $r$.
As $A_r$ is not trivial (and ${0 \rightarrow A}$ is epic)  Lemma~\ref{LemSubterImpliesCosimpleOP} implies that ${ A_r + A_r}$ is simple. Hence, by the computation in \cite[Proposition 7.3]{MundiciAdvanced}, $r$ must be rational.
\end{proof}

\section{The $\pi_0$ functor for topological spaces}
\label{s:stonereflection}

In this section we show that the full inclusion ${\Stone \rightarrow \KHaus}$ of the category of Stone spaces into that of compact Hausdorff spaces has a left adjoint  ${\pi_0\colon \KHaus \to \Stone}$ that preserves set-indexed products. The result just stated may be concisely referenced as follows. That the inclusion at hand  is reflective is well known and flows readily from the universal property of the quotient topology. As shown in \cite[Section~7]{CJKP97}, the reflection has ``stable units''; we need not discuss this property here, except to recall  that it easily implies that the left adjoint $\pi_0$ preserves finite products. Since Gabriel and Ulmer in \cite[p.\ 67]{GabrielUlmer}  show that $\pi_0$ preserves cofiltered limits, $\pi_0$ preserves all products.\footnote{We are grateful to Luca Reggio and to Dirk Hofmann for pointing out to us, respectively, the relevance of \cite[Section~7]{CJKP97} and of \cite[p.\ 67]{GabrielUlmer}.}

We give  here a different proof that emphasises the key r\^{o}le of totally disconnected spaces in the general case. We first obtain a product-preserving left adjoint to the full inclusion of the category $\TD$ of totally disconnected topological spaces into $\Top$. 
We then show how to restrict this left adjoint to the categories of interest to us in the present paper. 

A topological space $X$ is \emph{connected} if it so in the sense of Definition~\ref{DefConnected}. A subset of a space is \emph{clopen} if it is both closed and open. Then, a space $X$ is connected  if and only if  it contains exactly two clopen sets, which are then necessarily $\emptyset$ and $X$. Equivalently \cite[Theorem 6.1.1]{Engelking}, $X$ is connected if whenever $X=A\cup B$ with $A\cap B=\emptyset$ and $A$ and $B$ closed subsets of $X$, then exactly one of $A$ and $B$ is empty. If $X$ is a space and $x\in X$, the \emph{component} of $x$ in $X$, written $C_x$ (with $X$ understood), is defined as
\[
C_x\coloneqq\bigcup \{C\seq X\mid x \in X \text{ and } C \text{ is connected}  \} \subseteq X.
\]
It can be shown that $C_x$ is a connected subspace of $X$ \cite[Corollary 6.1.10]{Engelking}, and it therefore is the inclusion-largest such to  which $x$ belongs. Also, $C_x$ is closed in $X$ \cite[Corollary 6.1.11]{Engelking}.  A topological space $X$ is \emph{totally disconnected} if for each $x\in X$ we have $C_x=\{x\}$.

 Consider the equivalence relation on $X$ given by
\begin{align}\label{eq:connectedrel}
x\sim y \text{ if, and only if, } C_x=C_y,
\end{align}
and define
\[
\pi_0X\coloneqq\frac{X}{\sim}.
\]
We equip $\pi_0X$ with the quotient topology, and call it the \emph{space of components} of $X$. We write
\begin{align}\label{eq:pi0unit}
q\colon X \longrightarrow \pi_0X
\end{align}
for the quotient map.
\begin{lemma}\label{l:respectingcomp} For every continuous map ${f\colon X\to Y}$ between topological spaces there is exactly one map such that the square  below commutes.
\[
\xymatrix{
X \ar[d]_-{f} \ar[r] & \ar[d]^-{\pi_0f} \pi_0X\\
Y\ar[r] & \pi_0Y
}
\]
\end{lemma}
\begin{proof}
We first show that ${f\colon X\to Y}$ preserves the equivalence relation $\sim$ in \eqref{eq:connectedrel}. Given $x,x' \in X$, suppose $x\sim x'$, so that $C_x=C_y \eqqcolon C$. Since continuous maps preserve connectedness \cite[Theorem 6.1.3]{Engelking},  $f[C]$ is a connected subset of $Y$ that contains both $fx$ and $fx'$. Hence $f[C]\seq C_{fx}\cap C_{fx'}$, which entails $C_{fx}=C_{fy}$. This completes the proof that $f$ preserves $\sim$. Existence and uniqueness of ${\pi_0 f}$ follow from the universal property  of the quotient $X \rightarrow \pi_0 X$.
\end{proof}
\noindent Lemma~\ref{l:respectingcomp} implies that  the assignment
that sends $f$ to $\pi_0f$ extends to an endofunctor 
\begin{align}\label{eq:endotop}
\pi_0\colon \Top \longrightarrow \Top.
\end{align}
This endofunctor determines the full subcategory $\TD$, as we now show.

\begin{lemma}\label{LemClaimForl:td} 
If ${C \subseteq \pi_0 X}$ is a connected subspace then so is ${q^{-1} [C] \subseteq X}$. 
\end{lemma}
\begin{proof}
Let $q^{-1}[C]=F_1\cup F_2$ with $F_1$ and $F_2$ disjoint closed subsets of $X$. For any $y \in C$ we can write the fibre $q^{-1}[\{y\}]$ as $C_x$ for any   $x\in q^{-1}[\{y\}]$. Further, we can express $C_x$ as the disjoint union
$C_x=(F_1\cap C_x)\cup (F_2\cap C_x)$. And $C_x$ is closed and connected, because it is a component. Hence exactly one of $q^{-1}[\{y\}]=C_x\seq F_1$ or $q^{-1}[\{y\}]=C_x\seq F_2$ holds, for each $y \in C$. We can then define
\[
S_i\coloneqq\{ y \in C\mid q^{-1}[\{y\}] \seq F_i\}, \ i=1,2,
\]
to the effect that $C=S_1\cup S_2$ and $S_1\cap S_2 =\emptyset$. By construction we have $F_i=q^{-1}[S_i]$, $i=1,2$. The definition of quotient topology then entails that $S_i$ is closed because $F_i$ is. Since $C$ is connected, exactly one of $S_1$ and $S_2$ is empty, and hence so is exactly one of $F_1$ and $F_2$. 
\end{proof}

\begin{lemma}\label{l:adjointnessofpi0} For any space $X$,  the quotient map $q \colon X \rightarrow \pi_0X$ in \eqref{eq:pi0unit} is universal from
$X$ to the full inclusion ${\TD \rightarrow \Top}$. 
\end{lemma}
\begin{proof}
We first show that ${\pi_0 X}$  is totally disconnected.
Let $C_y$ be the component of ${y \in \pi_0 X}$, with the intent of showing it is a singleton. 
By Lemma~\ref{LemClaimForl:td}, since $C_y$ is connected in ${\pi_0 X}$, so is $q^{-1}[C_y]$  connected in $X$. Therefore $q^{-1}[C_y]$ is contained in the component $C_x$ of any $x\in X$ with $x\in q^{-1}[C_y]$; and thus, the direct image $q[q^{-1}[C_y]]$ is contained in $q[C_x]=\{y\}$. Since  $q[q^{-1}[C_y]]=C_y$, because $q$ is surjective, we conclude $C_y\seq \{y\}$, as was to be shown.

Let ${f\colon X\to Y}$ be a continuous map, with $Y$ totally disconnected.
We already know from the proof of Lemma~\ref{l:respectingcomp} that $f$ preserves $\sim$ so,
as $Y$ is totally disconnected, ${x \sim x'}$ in $X$ implies ${f x = f x'}$ in $Y$.
The universal property of the quotient ${q \colon X \rightarrow \pi_0 X}$ implies the existence of a unique ${g \colon \pi_0 X \rightarrow Y}$ such that ${g q = f}$.
\end{proof}
\noindent We conclude that the full inclusion ${\TD \rightarrow \Top}$ has a left adjoint that, with no risk of confusion, will again be denoted by ${\pi_0 \colon \Top \rightarrow \TD}$.

\begin{prop}\label{PropEngelking}
The functor ${\pi_0 \colon \Top\rightarrow \TD}$ preserves all set-indexed products.
\end{prop}
\begin{proof}
Consider  a family ${(X_s \mid s \in S)}$ of spaces in $\Top$ indexed by a set $S$ and let 
\[
{\gamma \colon \pi_0 \prod_{s\in S} X_s \longrightarrow \prod_{s\in S} \pi_0 X_s}
\]
be the unique map such that the  triangle below commutes
\[\xymatrix{
 \pi_0 \left( \prod_{s\in S} X_s\right) \ar[rd]_-{\pi_0 \pr_s} \ar[r]^-{\gamma} & \prod_{s\in S} \pi_0 X_s \ar[d]^-{\pr_s} \\
 & \pi_0 X_s 
}\]
for every ${s \in S}$. 
In other words,
${\gamma ( C( x_s \mid s\in S )) = (C x_s \mid s\in S) \in \prod_{s\in S} \pi_0 X_s }$
for any ${( x_s \mid s\in S )}$ in ${\prod_{s\in S} X_s}$.

To prove that $\gamma$ is injective assume that 
${\gamma ( q ( x_s \mid s\in S )) =\gamma ( q ( y_s \mid s\in S ))}$ in ${\prod_{s\in S} \pi_0 X_s}$. 
That is, ${q x_s = q y_s}$ in ${\pi_0 X_s}$ for every ${s \in S}$.
 By \cite[Theorem~{6.1.21}]{Engelking} we have 
${q ( x_s \mid s\in S ) = q ( y_s \mid s\in S )}$ in ${\pi_0 \left( \prod_{s\in S} X_s\right)}$, so $\gamma$ is injective.

To prove that $\gamma$ is surjective observe that the following diagram commutes
$$\xymatrix{
 & \ar[ld]_-{q} \prod_{s\in S} X_s  \ar[d]^-{\prod_{s\in S} q} \ar[r]^-{\pr_s} & X_s \ar[d]^-{q}  \\
\pi_0 \left( \prod_{s\in S} X_s\right)   \ar[r]_-{\gamma} & \prod_{s\in S} \pi_0 X_s  \ar[r]_-{\pr_s} & \pi_0 X_s
}$$
 for every ${s\in S}$, so the inner triangle commutes.
As products of  surjections are surjective, the inner vertical map is surjective and hence so is $\gamma$, the bottom map of the triangle.
\end{proof}

We next identify a related construction which will provide a useful alternative description of $\pi_0$ when restricted to $\KHaus$.
Let us write $\C{(X,\2)}$ for the set of continuous maps from the space $X$ to the discrete two-point space $\2\coloneqq\{0,1\}$. There is a canonical continuous function
\begin{align}
E = \langle f\mid f \in \C{(X,\2)}\rangle\colon &X\longrightarrow \2^{\C{(X,\2)}},\label{eq:E}\\
&x\longmapsto ( f x \mid f\in \C{(X,\2}) ).\nonumber
\end{align}
For any subset $S\seq X$, write $\chi_S\colon X\to \2$ for the characteristic function defined by $\chi_S x=1$ if, and only if,  $x\in S$.
Then $S$ is clopen precisely when $\chi_S\in \C{(X,\2)}$. Thus,  $E$ in \eqref{eq:E}
 can equivalently be described as the function that sends each point ${x \in X}$ to the set of clopen subsets of $X$ that contain $x$. 

In order to prove the next lemma recall \cite[p.\ 356]{Engelking} that the \emph{quasi-component} of $x \in X$ is defined as
 \[
 \widetilde{C}_x\coloneqq \bigcap\{S\seq X\mid S \text{ is clopen, and } x \in S\}.
 \]
 It is clear that the quasi-components of a space $X$ partition $X$ into closed non-empty sets.
The relation between $E$ and quasi-components may be stated as follows.

\begin{lemma}\label{LemEandQC} 
For any ${x, x' \in X}$, ${E x = E x'}$ if and only if ${\widetilde{C}_x = \widetilde{C}_{x'}}$.
\end{lemma}
\begin{proof}
If ${E x = E x'}$ then clearly ${\widetilde{C}_x = \widetilde{C}_{x'}}$.
For the converse assume that ${\widetilde{C}_x = \widetilde{C}_{x'}}$ and let ${S \subseteq X}$ be a clopen containing ${x}$. Then ${x' \in \widetilde{C}_{x'} = \widetilde{C}_x \subseteq S}$.
That is, ${x' \in S}$.
\end{proof}

The reader should beware that the quasi-component $\widetilde{C}_x$ of $x\in X$ in general fails to be connected. Indeed, the inclusion $C_x\seq \widetilde{C}_x$ always holds for each $x\in X$ \cite[Theorem 6.1.22]{Engelking}, and may be proper \cite[Example 6.1.24]{Engelking}. However:

\begin{lemma}\label{LemComparisonBetweenCandQC}
For any $X$ there exists a unique ${E' \colon \pi_0 X \rightarrow \2^{\C{(X,\2)}}}$ such that the following diagram
$$\xymatrix{
X \ar@(d,l)[rd]_-{E} \ar[r]^-q & \pi_0 X \ar[d]^-{E'} \\
 & \2^{\C{(X,\2)}}
}$$
commutes.
\end{lemma}
\begin{proof}
Let ${x, x' \in X}$ be such that ${x \sim x'}$; that is, ${C_x = C_{x'}}$.
Then 
\[ x \in C_x \cap C_{x'} \subseteq  \widetilde{C}_x \cap  \widetilde{C}_{x'} \]
 so, as quasi-components are equal or disjoint, ${\widetilde{C}_x =  \widetilde{C}_{x'}}$.
That is, ${E x = E x'}$ by Lemma~\ref{LemEandQC}.
\end{proof}

Let 
${\xymatrix{
X \ar[r]^-{D} & \pi'_0 X \ar[r]^-{m} & \2^{\C{(X,\2)}}
}}$
be the epi/regular-mono factorization of the canonical map $E$ in \eqref{eq:E}. Then the following square commutes
\[
\xymatrix{
X \ar[d]_-{D} \ar[r]^-{q} & \ar[d] \pi_0X \ar@{.>}[ld]|-{c} \ar[d]^-{E'} \\
\pi'_0X\ar[r]_-m & \2^{\C{(X,\2)}}
}\]
by Lemma~\ref{LemComparisonBetweenCandQC} and, as $q$ is regular-epi and $m$ is monic,   there is  exactly one continuous map ${c\colon\pi_0(X)\to\pi'_0(X)}$ making the inner-triangles commute.
Since $D$ is epic, so is $c$.
Also, since $m$ is a regular mono,  $\pi_0'X$ carries the subspace topology inherited from the product $\2^{\C{(X,\2)}}$ and, as the latter is a Stone space, ${\pi_0'X}$ is Hausdorff.

\begin{lemma}\label{LemXcompactSamePi0} If $X$ is compact Hausdorff then ${c \colon \pi_0 X \rightarrow \pi_0' X}$ is an isomorphism and these isomorphic spaces are Stone spaces.
\end{lemma}
\begin{proof}
First recall \cite[Theorem~{6.1.23}]{Engelking} that, in any compact Hausdorff space $X$, the equality $C_x=\widetilde{C}_x$ holds for each $x\in X$.
In other words, in this case, the function ${\pi_0 X \rightarrow \pi_0' X}$ is bijective.
Also, since $X$ is compact, so is ${\pi_0 X}$ because $q$ is surjective.
Hence, as we already know that ${\pi_0' X}$ is Hausdorff, the Closed Map Lemma implies that $c$ is an isomorphism.

Similarly, compactness of $X$ implies compactness of ${\pi_0' X}$ and hence, the Closed Map Lemma implies that $m$ is closed. Therefore,  ${\pi_0'X}$ is a closed subspace of the Stone space  $\2^{\C{(X,\2)}}$.
\end{proof}

It is classical that each Stone space is totally disconnected, so there is a full inclusion $\Stone\to\TD$ such that the following diagram
$$\xymatrix{
\Stone \ar[d] \ar[r] & \TD \ar[d] \\
\KHaus \ar[r] & \Top
}$$
commutes. Lemma~\ref{LemXcompactSamePi0} implies that the composite 
${\xymatrix{\KHaus \ar[r] & \Top \ar[r]^-{\pi_0} & \TD}}$
 factors  through the full inclusion ${\Stone \rightarrow \TD}$.
The factorization will be  conveniently denoted by ${\pi_0 \colon \KHaus \rightarrow \Stone}$.

\begin{theorem}\label{t:dualpierce} The functor $\pi_0\colon \KHaus \to \Stone$ is left adjoint to  the full inclusion $\Stone \to \KHaus$, and preserves all set-indexed products.\qed
\end{theorem}
\begin{proof}
Since, as observed above, $\pi_0\colon\Top \to \TD$ restricts to $\pi_0\colon\KHaus\to \Stone$, the fact that the former is a left adjoint to ${\TD \rightarrow \Top}$ (Lemma~\ref{LemXcompactSamePi0}) restricts to the fact that  $\pi_0\colon\KHaus\to \Stone$ is left adjoint to ${\Stone \rightarrow \KHaus}$.
 It is standard that products in $\KHaus$ and in $\Stone$ agree with products in $\Top$ (using, in particular,  Tychonoff's Theorem that any product of compact spaces is compact), so  Proposition \ref{PropEngelking} entails that $\pi_0\colon\KHaus\to \Stone$ preserves all set-indexed products.
\end{proof}

\section{Spectra of MV-algebras}\label{s:Spectra}

In this section we recall  the material about spectra of MV-algebras that is needed in the sequel.

Recall that an ideal $\p$ of an MV-algebra $A$ is \emph{prime} if it is proper, and  the quotient $A/\p$ is totally ordered. The (\emph{prime}) \emph{spectrum} of an MV-algebra $A$ is 
\[
\Spec{A}\coloneqq\{\p\subseteq A\mid \p \text{ is a prime ideal of } A\}
\]
 topologised into the \emph{spectral space} of $A$, as follows. For  a subset $S\seq A$, define 
\begin{align*} 
\V{(S)}&\coloneqq\{\p \in \Spec{A}\mid S\subseteq \p\},\\
\Su{(S)}&\coloneqq \Spec{A}\setminus\V{(S)}=\{\p \in \Spec{A}\mid S\not\subseteq \p\}.
\end{align*}
The set $\V{(S)}$ is called the \emph{vanishing locus}, or \emph{zero set}, of $S$, while $\Su{(S)}$ is called its \emph{support}. If $a \in A$, write $\V{(a)}$ as a shorthand for $\V{(\{a\})}$, and similarly for $\Su{(a)}$.  Then the collection
 \begin{align*} 
 \{\V{(I)}\mid I \text{ is an ideal of } A\}
 \end{align*}
 is the set of closed sets for a topology on $\Spec{A}$ that makes the latter a spectral space in the sense of Hochster \cite{Hochster69}. The collection
 \begin{align*} 
 \{\Su{(a)}\mid  a\in A\}
 \end{align*}
is  a basis of compact open sets for this topology; see \cite[Chapitre 10]{BKW} and \cite[Chapter 4]{MundiciAdvanced}.  The topology is variously known as   the \emph{Stone},  \emph{Zariski}, or \emph{hull-kernel} topology of $A$.

The assignment ${A \mapsto \Spec{A}}$ extends to a functor ${\opCat{\MV} \to \Top}$, because inverse images of primes ideals along homomorphisms are prime. Althouh it is common to take the codomain of $\Spec$ as the category of spectral spaces and spectral maps, for our purposes in this paper it is expedient to regard $\Spec$ as taking values in $\Top$.

 The \emph{maximal spectrum} of an MV-algebra $A$ is
\[
\Max{A}\coloneqq\{\m\subseteq A\mid \m \text{ is a maximal ideal of } A\}.
\]
We have $\Max{A}\seq \Spec{A}$, or equivalently, any simple MV-algebra is totally ordered (see e.g.\ \cite[3.5.1]{CignoliEtAlBook}).  
  The \emph{maximal spectral space} of $A$ is the set $\Max{A}$ equipped with the subspace topology it inherits from $\Spec{A}$. Then $\Max{A}$ is a compact Hausdorff space \cite[Proposition 4.15]{MundiciAdvanced}, and every compact Hausdorff space arises in this manner from some MV-algebra $A$ \cite[Theorem 4.16(iv)]{MundiciAdvanced}.

The standard example of MV-algebra, the interval
$[0,1]$  equipped with the constant $0$ and the operations $\oplus$, $\neg$, generalises as follows. If $X$ is any set, the collection $[0,1]^{X}$ of all functions from $X$ to $[0,1]$ inherits the structure of an MV-algebra upon defining operations pointwise. If, additionally, $X$ is a topological space, since $\oplus\colon [0,1]^{2}\to [0,1]$, $\neg\colon[0,1]\to[0,1]$, and $0$ are continuous with respect to the Euclidean topology of $[0,1]$, the subset 
\begin{align}\label{eq:c(x)}
\C(X)\coloneqq\{f\colon X\to [0,1]\mid f \,\text{ is continuous}\}
\end{align}
is a subalgebra of the MV-algebra $[0,1]^{X}$. We shall describe a natural MV-ho\-mo\-mor\-phism $\eta_A\colon A \longrightarrow \C{(\Max{A})}$, for each MV-algebra $A$. Its existence descends from  H\"older's Theorem (Lemma \ref{MVHolder}), which allows us to define a close analogue to the Gelfand transform in functional analysis. Indeed, in light of that result, to  $a\in A$ and $\m\in\Max{A}$  we  associate the real number $\h_\m(a / \m)\in [0,1]$, obtaining the function
\begin{align}\label{eq:transform}
\widehat{a}\colon \Max{A}&\longrightarrow [0,1] \\
\m &\longmapsto h_{\m}(\tfrac{a}{\m}).\nonumber
\end{align}
It can be shown \cite[4.16.iii]{MundiciAdvanced} that the function \eqref{eq:transform} is continuous with respect to the Stone topology of $\Max{A}$ and the Euclidean topology of $[0,1]$.
We thereby arrive at the announced homomorphism
\begin{align}
\eta_{A} \colon A &\longrightarrow \C(\Max{A})\label{eq:unit}\\
a&\longmapsto \widehat{a}\nonumber
\end{align}
for each MV-algebra $A$.

\begin{lemma}\label{mv6} 
For any MV-homomorphism $h\colon A\to B$ and any $\m\in\Max{B}$ we have $h^{-1}(\m)\in\Max{A}$. Moreover, the inverse-image map $h^{-1}\colon \Max{B}\to\Max{A}$ is continuous with respect to the Stone topology.
\end{lemma}
\begin{proof}The first assertion is proved in \cite[1.2.16]{CignoliEtAlBook}. The second assertion is a straightforward verification using the definition of Stone topology.
\end{proof}
\noindent In light of Lemma \ref{mv6} we henceforth regard $\Max$ as a functor:
\begin{align}\label{eq:maxfunctor}
\Max \colon \MV \longrightarrow \KHaus^{\rm op},
\end{align}
where $\KHaus$ denotes the category of compact Hausdorff spaces and their continuous maps. 

 Given a continuous map $f\colon X \to Y$ in $\KHaus$, it is elementary that the induced function 
\begin{align*}
\C(f)\colon \C(Y)&\longrightarrow\C(X),\\
g\in \C(Y)&\longmapsto g\circ f \in \C(X)
\end{align*}
is a morphism in \MV. We therefore regard $\C$ as a functor:
\begin{align*}
\C \colon  \KHaus^{\rm op}\longrightarrow  \MV. 
\end{align*}

There is an adjunction  
\begin{align}\label{eq:CDMadj}
\Max\dashv\C\colon \KHaus^{\rm op}\to \MV
\end{align}
known as the \emph{Cignoli-Dubuc-Mundici} adjunction \cite{CDMadj}; see \cite[Section 3]{MR} for further references and details not mentioned below.
Dually to \eqref{eq:unit}, for any  space $X$ in $\KHaus$ there is a continuous map
\begin{align}
\epsilon_{X} \colon X &\longrightarrow \Max{\C(X)}\label{eq:counit}\\
x&\longmapsto \{f\in\C(X)\mid f(x)=0\},\nonumber
\end{align}
and it is a standard fact that $\epsilon_X$ is a homeomorphism. (Compare \cite[4.16]{MundiciAdvanced}.)
Writing $\id{{\sf C}}$ for the identity functor on a category {\sf C}, we  can summarise the adjunction as follows.
\begin{theorem}[{\cite[Propositions 4.1 and 4.2]{CignoliEtAlBook}}]\label{t:Max-C-adj}
The functor $\Max$ is left adjoint to the fully faithful $\C$, i.e.\ $\Max\dashv\C\colon \KHaus^{\rm op}\to \MV$. The unit and the counit of the adjunction are the natural transformations $\eta \colon \id{\MV} \to \C\Max$ and $\epsilon \colon \Max\C \to \id{\KHaus^{\rm op}}$ whose components are given by \eqref{eq:unit} and \eqref{eq:counit}, respectively.\qed
\end{theorem}

\section{The Pierce functor preserves coproducts}\label{s:Pierce}

The category $\BA$ of Boolean algebras may be identified with the domain of the full subcategory ${\I \colon \BA \rightarrow \MV}$ determined by the MV-algebras whose operation $\oplus$ is idempotent. It is then clear that ${\I \colon \BA \rightarrow \MV}$ is a variety so, in particular, it has a left adjoint.
It also has a right adjoint that we now describe.

We write $\PP{A}$ for the collection of all Boolean elements of the MV-algebra $A$. By \cite[1.5.4]{CignoliEtAlBook}, $\PP{A}$ is the largest subalgebra of $A$ that is a Boolean algebra. A homomorphism $h\colon A\to B$ preserves Boolean elements, because the latter are defined by equational conditions. Therefore, $h$ induces by restriction a function $\PP{h}\colon \PP{A}\to \PP{B}$ that is evidently a homomorphism of Boolean algebras. We thus obtain a functor 
\[
\PP\colon \MV \longrightarrow \BA
\]
from the category of MV-algebras to that of Boolean algebras; we call it the \emph{Pierce functor} because of the close analogy with the theory developed in \cite{Pierce67} for rings.  
 \begin{lemma}\label{l:PierceRight}The functor $\PP$ is right adjoint to the functor $\I$.
 \end{lemma}
 \begin{proof}This is a direct consequence of the fact that $\PP{A}$ is the largest Boolean subalgebra of $A$, for any MV-algebra $A$. 
 \end{proof}

The proof of  Proposition~\ref{PropMVisCoextensive}---in particular, Lemma \ref{l:productsplitting}---makes it clear that ${\PP\colon \MV \rightarrow \BA}$ is essentially the `complemented subobjects' functor $\rmB$ determined by the extensive category ${\opCat{\MV}}$.

We now embark on the proof of the central fact that ${\PP \colon \MV \rightarrow \BA}$ preserves  coproducts. Our aim is to reduce the problem to a situation where we can apply the topological results in Section \ref{s:stonereflection}.

\begin{lemma}\label{l:booleanelement} For any MV-algebra $A$ and any element $a\in A$, $a$ is Boolean if, and only if, for each  prime ideal $\p$ of $A$, we have $a/\p\in \{0,1\}\seq A/\p$.
\end{lemma}
\begin{proof}Let $C$ be any totally ordered MV-algebra. For $x\in C$, either $x\leq \neg x$ or $\neg x \leq x$. If the former holds then $x\wedge \neg x=x$, so that if $x$ is Boolean then $x=0$. If the latter holds then $x\vee \neg x=x$, and thus $x=1$ if $x$ is Boolean. In summary, if $x\in C$ is Boolean then $x\in \{0,1\}$. The converse implication is clear. Summing up, the Boolean elements of $C$ are precisely $0$ and $1$.

Boolean elements, being definable by equational conditions, are preserved by homomorphisms. Hence if $a$ is Boolean then $a/\p\in A/\p$ is Boolean, and therefore, since $A/\p$ is totally ordered, $a/\p\in\{0,1\}$ by the argument in the preceding paragraph. This proves the left-to-right implication in the statement of the lemma.
For the converse implication, we recall that in any MV-algebra $A$ we have $\bigcap \Spec{A}=\{0\}$ \cite[1.3.3]{CignoliEtAlBook}. Hence, the function $\iota \colon A \longrightarrow \prod_{\p\in \Spec{A}} A/\p$ defined by $a \in A \longmapsto (a/\p)_{\p \in \Spec}\in \prod_{\p\in \Spec{A}} A/\p$ is an injective homomorphism. Assume that for each $\p\in\Spec{A}$ we have $a/\p\in\{0,1\}$. Since operations in $\prod_{\p\in \Spec{A}} A/\p$ are computed pointwise, we infer $\iota(a)\vee\neg\iota(a)= (a/\p)_{\p \in \Spec}\vee \neg (a/\p)_{\p \in \Spec}=1$, and therefore, since $\iota$ is an isomorphism onto its range, $a\vee \neg a=1$. This completes the proof.
\end{proof}

\begin{lemma}\label{l:specreflectssums} Let $A$ be an MV-algebra, and suppose there exist \textup{(}possibly empty\textup{)}  closed subsets $X_0,X_1\seq \Spec{A}$ with $\Spec{A}=X_0\cup X_1$ and $X_0\cap X_1=\emptyset$. Then there exists exactly one Boolean element $b\in A$ such that $b/\p=0$ for each $\p\in X_0$ and $b/\p=1$ for each  $\p\in X_1$.
\end{lemma}
\begin{proof}By \cite[10.1.7]{BKW}, there is exactly one ideal $I_i$ of $A$ such that $\V{(I_i)}=X_i$, $i=0,1$. Consider the elements $0,1\in A$. The fact that $\Spec{A}$ is partitioned into $X_0$ and $X_i$  entails  $I_0\vee I_1=A$ and $I_0\cap I_1=\{0\}$, so that the Chinese Remainder Theorem \cite[Lemme 10.6.3]{BKW} applied to $0$ and $X_0$, and to $1$ and $X_1$, yields   one element $b\in A$ such that $b/I_0=0$ and $b/I_1=1$. Using the Third Isomorphism Theorem, the latter conditions imply $b/\p\in \{0,1\}$ for each $\p\in \Spec{A}$ so that by Lemma \ref{l:booleanelement}  we conclude that $b$ is Boolean. If $b'\in A$ also satisfies $b'/\p=0$ for each $\p\in X_0$ and $b'/\p=1$ for each  $\p\in X_1$, then $b/\p=b'/\p$ for $\p\in\Spec{A}$, so that $b=b'$ because  $\bigcap \Spec{A}=\{0\}$ \cite[1.3.3]{CignoliEtAlBook}.
\end{proof}

We record a corollary that will have further use in the paper. It is the exact analogue for MV-algebras of a standard result for the category $\Ring$, see e.g.\ \cite[Theorem 7.3]{Jacobson-II}. In order to state it, let us  write $\Cp{X}$ for the Boolean algebra of clopen sets of any topological space  $X$. Let us then observe that the uniqueness assertion about the Boolean element $b$ in Lemma~\ref{l:specreflectssums} allows us to define, for any MV-algebra $A$, a function 
\begin{align}\label{eq:chi}
\chi_A\colon\Cp(\Spec{A})\longrightarrow {\PP}A
 \end{align}
 that assigns to each $X_0\in \Cp{(\Spec{A})}$ the unique element $b\in{\PP}A$ with the properties stated in that lemma  with respect to $X_0$ and $X_1\coloneqq \Spec{A}\setminus X_0$. It is then elementary to verify that $\chi_A$ is a homomorphism of Boolean algebras.

\begin{corollary}\label{c:jacobson}For any MV-algebra $A$, the function
\[
\phi_A\colon {\PP}A \longrightarrow\Cp(\Spec{A})
\]
that sends $b\in {\PP}A$ to $\V{(b)}\seq \Cp(\Spec{A})$ is  an isomorphism of Boolean algebras whose inverse is the homomorphism $\chi_A$ in \eqref{eq:chi}. In particular, $A$ is  indecomposable if, and only if, $\Spec{A}$ is connected.
\end{corollary}
\begin{proof}
By Lemma~\ref{l:booleanelement}  it is clear that $\V{(b)}$ for each $b\in{\PP}A$ is clopen and that $\phi_A$ is a homomorphism. Let us consider $b'\coloneqq\chi_A\phi_Ab$. For each $\p\in \V{(b)}$ we have $b/\p=0$ by definition of $\V$, and $b'/\p=0$ by the defining property of $b'$.  Similarly, for each $\p\in\Spec{A}\setminus\V{(A)}$ we have $b/\p=b'/\p=0$. Thus, $b$ and $b'$ agree at each prime and thus $b=b'$ because $\bigcap\Spec{A}=\{0\}$ \cite[1.3.3]{CignoliEtAlBook}. Conversely, for $X_0\in \Cp{(\Spec{A})}$, consider the clopen $\phi_A\chi_AX_0$. For $\p\in \Spec{A}$, by definition of $\chi_A$ we have $\p \in X_0$ if, and only if, $(\chi_AX_0)/\p=0$. Hence
$\phi_A (\chi_A X_0)=X_0$, and the proof is complete.
\end{proof}
The \emph{radical} of $A$ is the ideal
\[
\Rad{A}\coloneqq\bigcap\Max{A}.
\]
In accordance with standard terminology in general algebra, one says $A$ is \emph{semisimple} precisely when $\Rad{A}=\{0\}$.
We note in passing that, unless $A$ is semisimple, the statement in Lemma \ref{l:booleanelement} cannot be strenghtened to ``$a$ is Boolean if, and only if, for each  $\m\in \Max{A}$ we have $a/\m\in \{0,1\}\seq A/\m$''. 
\begin{lemma}\label{l:allbooleanarein}
Let $A$ be an MV-algebra, and suppose there exist \textup{(}possibly empty\textup{)}  closed subsets $X_0,X_1\seq \Max{A}$ with $\Max{A}=X_0\cup X_1$ and $X_0\cap X_1=\emptyset$. Then there exists exactly one Boolean element $b\in A$ such that $b/\m=0$ for each $\m\in X_0$ and $b/\m=1$ for each  $\m\in X_1$.
\end{lemma}
\begin{proof}By \cite[1.2.12]{CignoliEtAlBook}, each $\p \in \Spec{A}$ is contained in exactly one $\lambda{\p} \in \Max{A}$, so that we can define a function
\begin{align}\label{eq:lambda}
\lambda\colon \Spec{A}&\longrightarrow \Max{A},\\
 \p &\longmapsto {\lambda}\p.\nonumber
\end{align} 
By \cite[10.2.3]{BKW}, this function is  continuous, and  it is a retraction for the inclusion $\Max{A}\seq\Spec{A}$. Therefore, $X'_0\coloneqq \lambda^{-1}[X_0]$ and $X'_1\coloneqq \lambda^{-1}[X_1]$ are closed subsets of $\Spec{A}$ satisfying $\Spec{A}=X'_0\cup X'_1$ and $X'_0\cap X'_1=\emptyset$. Now  Lemma \ref{l:specreflectssums} provides a unique Boolean element $b$ such that $b/\p=0$ for each $\p\in X_0'$, and $b/\p=1$ for each $\p \in X_1'$. As $X_i\seq X_i'$, $i=0,1$, $b$ satisfies the condition in the statement. Concerning uniqueness, suppose $a$ is a Boolean element of $A$ such that $a/\m=0$ for each $\m\in X_0$, and $a/\m=1$ for each $\m \in X_1$. We claim $a=b$. Indeed, let $\p \in X_i'$, $i=0,1$. Then $a/\lambda{\p}=i$ because $\lambda \p\in X_i$. The inclusion $\p\seq \lambda\p$ induces a quotient map $q\colon A/\p\to A/\lambda\p$. By Lemma \ref{l:booleanelement} we have $a/\p \in \{0,1\}$. Also, $A/\lambda\p$ is nontrivial. Therefore since $q(a/\p)=a/\lambda\p=i$ it follows that $a/\p=i$. By the uniqueness assertion in Lemma \ref{l:specreflectssums} we conclude $a=b$.
\end{proof}

\begin{remark}\label{r:lambda}
We observe that the analogue of Lemma \ref{l:allbooleanarein} about coproduct decompositions of $\Max{A}$ being indexed by idempotent elements does not hold in general for rings. Indeed, spectra of MV-algebras always are completely normal---which affords the existence of the map $\lambda$ used in the proof above---whereas spectra of rings are not, in general. 
For more on the important r\^{o}le that the continuous retraction $\lambda$ in \eqref{eq:lambda} plays in the theory of lattice-groups and MV-algebras, see \cite{BallMarraMcNeillPedrini} and the references therein.
\end{remark}

Our next objective is to show that  $\PP$ sends the unit $\eta$ of ${\C \dashv \Max}$ in \eqref{eq:unit} to an isomorphism.

\begin{lemma}\label{l:PierceOfUnit} For any MV-algebra $A$, the morphism $\PP \eta_A\colon \PP A\rightarrow (\PP \C \Max){A}$ is an isomorphism.
\end{lemma}
\begin{proof}Let $b'\in \C{(\Max{A})}$ be Boolean, with the aim of exhibiting $b\in \PP{A}$ such that $\eta_A(b)=b'$. Evaluating the defining equality $b'\oplus b'=b'$ at each $\m\in\Max{A}$ we see that $b'(\m)\in \{0,1\}$ holds. Therefore, the two  closed subsets $X_0\coloneqq b'^{-1}[\{0\}]$ and  $X_1\coloneqq b'^{-1}[\{1\}]$ of $\Max{A}$ satisfy the hypotheses of Lemma \ref{l:allbooleanarein}. We conclude that there exists one Boolean element $b\in A$ with $b/\m=0$ for $\m\in X_0$ and $b/\m=1$ for $\m \in X_1$. By the definition of $\eta_A$ this entails at once $\eta_A(b)=b'$, so $\eta_A$ is surjective. By the uniqueness statement in Lemma \ref{l:allbooleanarein}, $\eta_A$ is also injective.
\end{proof}
Our next step will be to factor $\PP$ into a manner that is useful to our purposes. 
Lemma \ref{l:PierceOfUnit} implies that the functors $\MV\to\BA$ in the diagram below 
\begin{align}\label{eq:diagPierce}
\xymatrix{
\MV \ar[d]_-{\Max} \ar[r]^-{\PP}& \BA \\
\KHaus^{\rm op}\ar[r]_-{{\C}} & \MV \ar[u]_-{{\PP}}
}
\end{align}
are naturally isomorphic.
\begin{lemma}\label{l:co-pi0}The functor ${{\PP}{\C}\colon \KHaus^{\rm op} \to \BA}$  preserves all set-indexed coproducts.
\end{lemma}
\begin{proof}Using Stone duality, it is an exercise to verify that the composite functor ${{\PP}{\C}\colon \KHaus^{\rm op}\to \BA}$ induces, by taking opposite categories on each side, a functor naturally isomorphic to the functor $\pi_0\colon \KHaus\to \Stone$ of Section \ref{s:stonereflection}. The lemma then follows from Theorem \ref{t:dualpierce}.
\end{proof}

We finally obtain the main result of this section.

\begin{theorem}\label{t:Pierce}The Pierce functor $\PP\colon \MV\to \BA$ preserves all set-indexed coproducts.
\end{theorem}
\begin{proof}As we saw above, the triangle \eqref{eq:diagPierce} commutes up to a natural isomorphism. Further, $\Max$  preserves arbitrary set-indexed colimits because it is left adjoint by Theorem \ref{t:Max-C-adj}; and
  ${\PP}{\C}$ preserves set-indexed coproducts by Lemma \ref{l:co-pi0}. Hence $\PP$ preserves set-indexed coproducts.
\end{proof}

\section{Main result, and final remarks}\label{s:main}

Let $\calA$ be a coextensive category.
Recall from the introduction that an object $A$ in  $\calA$ is {\em separable} if $A$ is decidable as an object in the extensive ${\opCat{\calA}}$. Thus, $A$ is separable if, and only if, there is a morphism ${f \colon A + A \rightarrow A}$ such that the span
$$\xymatrix{
A & \ar[l]_-{\nabla} A + A \ar[r]^-{f} & A 
}$$
is a product diagram. 

\begin{theorem}\label{ThmMain} Separable MV-algebras coincide with finite products of subalgebras of $[0,1]\cap\Q$.
\end{theorem}
\begin{proof}
By Theorem~\ref{t:Pierce} we have an  reflection ${\pi_0 \dashv \opCat{\I} \colon \Stone \rightarrow \opCat{\MV}}$ such that both adjoints preserve finite products and finite coproducts, so Proposition~\ref{PropMain} implies that every decidable object in ${\opCat{\MV}}$ is a finite coproduct of subterminal objects. Theorem~\ref{t:superseparableNew} completes the proof.
\end{proof}
We conclude the paper with some final remarks  that point to further research aimed at developing  an ‘arithmetic connected-component functor’. 
The guiding result from  Algebraic Geometry is this: the category $\calE$ of \'etale schemes over $K$ is reflective as a subcategory of that of  locally algebraic schemes over $K$ \cite[Proposition~{I, \S 4, 6.5}]{DemazureGabriel}. The left adjoint there is denoted by $\pi_0$, and  ${\pi_0 X}$ is called the {\em $k$-sch\'ema des composantes connexes de $X$} 
in Definition~{I, \S 4, 6.6} op.~cit. Moreover, it is then proved that ${\pi_0}$ preserves finite coproducts.
In terms of extensive categories, this says that for ${\calC = \opCat{\calE}}$,  the subcategory ${\Dec{\calC} \rightarrow \calC}$ has a finite-product preserving left adjoint.
We announce that the same holds for ${\calC = \opCat{\MV_{\rm fp}}}$, where $\MV_{\rm fp}$ is category of finitely presetable MV-algebras. The proof will be published elsewhere, but it is appropriate to indicate here   the r\^{o}le of locally finite MV-algebras  in connection with that result.

An MV-algebra $A$ is \emph{locally finite} if each finitely generated subalgebra of $A$ is finite. Finite MV-algebras are evidently locally finite; $[0,1]\cap \Q$ is an example of a locally finite MV-algebra that is not finite. Locally finite MV-algebras were studied in \cite{CDMadj}; see also \cite{CM} for a generalisation of the results in
\cite{CDMadj}, and \cite[Section 8.3]{MundiciAdvanced} for further material and \cite{AbbadiniSpada2022} for recent progress on the topic.
The connection with Theorem~\ref{t:superseparableNew} is the following characterisation of rational algebras.
\begin{lemma}\label{l:rational}
For any  MV-algebra $A$ the following are equivalent.
\begin{enumerate}[\textup{(}i\textup{)}]
\item\label{i:rat2} $A$ is simple and locally finite.
\item\label{i:rat3} $A$ is a subalgebra of $[0,1]\cap\Q$.
\end{enumerate}
\end{lemma}
\begin{proof}$\eqref{i:rat2}\Rightarrow\eqref{i:rat3}$. By H\"older's Theorem (Lemma \ref{MVHolder}), since $A$ is simple there is exactly one monomorphism $A\to [0,1]$; let us therefore identify $A$ with a subalgebra of $[0,1]$. If $A$ contains an irrational number $\rho\in [0,1]$ then the subalgebra generated by $\rho$ is infinite. Indeed, the  Euclidean algorithm of successive subtractions applied  to $\rho,1\in \R$ does not terminate (because $\rho$ and $1$ are incommensurable) and produces an infinite descending sequence of distinct, non-zero elements of $A$. Thus, $A \subseteq [0,1]\cap\Q$ by local finiteness.

\noindent $\eqref{i:rat3}\Rightarrow\eqref{i:rat2}$. Any subalgebra of $[0,1]$  evidently has no proper non-trivial ideal, by the Archimedean property of the real numbers, and is therefore simple. If, moreover, $A\subseteq [0,1]\cap\Q$,    the subgroup of $\R$ generated by finitely many $a_{1},\ldots,a_{n}\in A$ together with $1$ is discrete, and therefore by \cite[3.5.3]{CignoliEtAlBook} the subalgebra generated by $a_{1},\ldots,a_{n}$ is a finite chain. Thus $A$ is locally finite.
\end{proof}
\begin{corollary}\label{CorCharSepInTermsOfLF} An MV-algebra $A$ is separable if, and only if, $A$ is locally finite and ${\PP\! A}$ is finite.
\end{corollary}
\begin{proof}
If $A$ is separable then, by Theorem~\ref{ThmMain}, ${A = \prod_{i\in I} A_i}$ with $I$ finite and  ${A_i \subseteq   [0,1]\cap\Q}$ for each ${i \in I}$.
In particular, ${\PP A}$ is finite.
Also, each $A_i$  is locally finite by Lemma~\ref{l:rational}.
As finite products of locally finite algebras are locally finite,  $A$ is locally finite.
Conversely, assume that $A$ is locally finite and ${\PP A}$ is finite.
Then, ${A = \prod_{i\in I} A_i}$ with $I$ finite and  $A_i$  directly indecomposable for each ${i \in I}$.
As locally finite algebras are closed under quotients,  each $A_i$  is locally finite.
Hence, each $A_i$ is locally finite and  indecomposable.
But then $A$ must be simple. Indeed, Corollary \ref{c:jacobson} entails that $\Spec{A}$ is connected, and $\Spec{A}=\Max{A}$ by \cite[Theorem 5.1]{CDMadj}. Then the spectral space $\Spec{A}$ is Hausdorff, and thus has a base of clopen sets---hence, being compact, it is a Stone space. Since Stone spaces are  totally disconnected, connectedness of $\Spec{A}$  entails that $\Spec{A}$ is a singleton, so $A$ has exactly two ideals, and so is simple.
By Lemma~\ref{l:rational}, $A$ is then a subalgebra of  ${[0,1]\cap\Q}$. Therefore,  $A$ is separable by  Theorem~\ref{ThmMain}.
\end{proof}
Now, let ${\LF \rightarrow \MV}$ be the full subcategory determined by locally finite MV-algebras.
 Let us prove that this subcategory  is coreflective.

An element $a$ of an MV-algebra $A$ is \emph{of finite order-rank}\footnote{The terminology we introduce here is best motivated using lattice-groups---please see Appendix \ref{a:l-groups}.} if the subalgebra $B$ it generates in $A$ is finite. If $B$ is terminal, we say the order-rank of $a$ is zero. Otherwise, there exists exactly one $n\in \{1,2,\ldots\}$ such that $B=C_1\times \cdots\times C_n$ with each $C_i$ directly indecomposable and non-terminal, and we then say the order-rank of $a$ is $n$. We set
\[
{\RR}A\coloneqq\{a \in A \mid a \text{ is of finite order-rank}\}.
\]
Note that ${\PP}A\seq {\RR}A$, because any Boolean algebra is locally finite.
For any MV-algebra $A$ and subset $G\seq A$, let us write ${\Sa}{G}$ for the subalgebra of $A$ generated by $G$. When $G=\{g\}$ we write ${\Sa}g$ for ${\Sa}\{g\}$. 
\begin{lemma}\label{l:Risafunctor}Any homomorphism of MV-algebras sends elements of finite order-rank to elements of finite order-rank.
\end{lemma}
\begin{proof}Let $h\colon A\to B$ be a homomorphism and let $a \in {\RR}A$. Since $h$ commutes with operations, a routine argument in general algebra shows that $h[{S}a]=\Sa{(ha)}$; since ${\Sa}a$ is finite, so is $\Sa{(ha)}$.
\end{proof}
\begin{lemma}\label{l:Rissubalgebra}For any MV-algebra $A$, ${\RR}A$ is a locally finite subalgebra of $A$. Further, ${\RR}A$ is the inclusion-largest locally finite subalgebra of $A$.
\end{lemma}
\begin{proof}Let $F\coloneqq\{a_1,\ldots,a_n\}\seq A$ be a finite subset of elements of finite order-rank, $n\geq 0$ an integer. We need to show that the subalgebra ${\Sa}F$ of $A$ generated by $F$ is finite. Induction on $n$. If $n=0$ then ${\Sa}\emptyset$ is either the terminal one-element algebra or the initial two-element algebra. Now suppose $G\coloneqq\{a_1,\ldots, a_{n-1}\}$ is such that ${\Sa}G$ is finite. The subalgebra ${\Sa}a_n$ is also finite, because $a_n$ is of finite order-rank by hypothesis. The subalgebra ${\Sa}F$ is the least upper bound of ${\Sa}G$ and of ${\Sa}a_n$ in the lattice of subalgebras of $A$, and therefore can be written as a quotient of the coproduct  ${\Sa}G+{\Sa}a_n$. In more detail, by the universal property of the coproduct, the inclusion maps ${\Sa}G\seq {\Sa}F$ and ${\Sa}a_n\seq {\Sa}F$ induce a unique homomorphism $h\colon  {\Sa}G+{\Sa}a_n\to A$ whose regular-epi/mono factorisation $h=m q$ is such that $m\colon S\to A$ exhibits the subobject of $A$ that is the join of the subobjects ${\Sa}G$ and ${\Sa}a_n$---in particular, $S$ is isomorphic to ${\Sa}F$. So ${\Sa}F$ is a quotient of the  algebra ${\Sa}G+{\Sa}a_n$. Since finite coproducts of finite MV-algebras are finite by \cite[Corollary 7.9(iii)]{MundiciAdvanced}, ${\Sa}G+{\Sa}a_n$ is finite and therefore so is ${\Sa}F$.

To show that ${\RR}A$ is a subalgebra of $A$, first note that clearly $0\in{\RR}A$. If $a\in {\RR}A$ then $\neg a$ lies in the subalgebra generated by $a$, which is finite;  hence $\neg a$ is of finite order-rank. If $a,b \in {\RR}A$, then $a\oplus b$ lies in the subalgebra generated by $\{a,b\}$, which is finite by the argument in the preceding paragraph; hence $a\oplus b$ is of finite order-rank.

For the last assertion in the statement, let $B$ be a locally finite subalgebra of $A$. Given any $b \in B$, the subalgebra generated by $b$ in $A$ is finite, by our assumption about $B$; hence $b$ is of finite order-rank, and ${b\in \RR A}$. This completes the proof.
\end{proof}
Lemmas \ref{l:Risafunctor} and \ref{l:Rissubalgebra} allow us to regard $\RR$ as a functor
\[
\RR\colon \MV\longrightarrow \LF.
\]
\begin{corollary}\label{c:rightadjlf}
The functor $\RR\colon \MV\longrightarrow \LF$ is right adjoint to the full inclusion $\LF\longrightarrow \MV$.
\end{corollary}
\begin{proof}This is an immediate consequence of the fact that ${\RR}A$ is the largest locally finite subalgebra of the MV-algebra $A$, as proved in Lemma \ref{l:Rissubalgebra}.
\end{proof}

\begin{remark}\label{r:productslocfin}It is proved in \cite[Theorem 8.10]{MundiciAdvanced} that $\LF$ has all set-indexed products. This follows at once from Corollary \ref{c:rightadjlf}: indeed, for any set-indexed family $\{A_i\}_{i \in I}$ of locally finite MV-algebras the product of  $\{A_i\}_{i \in I}$ in $\LF$ is the coreflection $\RR{(\prod_{i \in I}A_i)}$ of the product $\prod_{i \in I}A_i$ in $\MV$.
\end{remark}

 We have been unable to prove that ${\opCat{\RR} \colon \opCat{\MV} \rightarrow \opCat{\LF}}$ preserves finite products. However,  writing ${\calC}$ for ${ \opCat{\MV_{\rm fp}}}$, we can show that the functor ${\opCat{\RR} }$ restricts to a left adjoint ${\pi_0 \colon \calC \rightarrow \Dec \calC}$ to the inclusion
 ${\Dec\calC \rightarrow \calC}$ and, moreover, it  preserves finite products.
As mentioned, the proof  will appear elsewhere.
 
\appendix\section{Separable unital lattice-ordered Abelian groups} \label{a:l-groups}
For background on lattice-groups we refer to \cite{BKW}. We recall that a \emph{lattice-ordered group}, or \emph{$\ell$-group} for short, is a group that is also a lattice\footnote{In this appendix, lattices are only required to have binary meets and joins, but not  top or bottom elements.} such that the group operation distributes over binary meets and joins. We only consider Abelian $\ell$-groups, and thus adopt additive notation. The underlying group of an Abelian $\ell$-group is torsion-free, and its underlying lattice is distributive. Write $\lA$ for the category of Abelian $\ell$-groups and of lattice-group homomorphisms. An element $1\in G$ in an Abelian $\ell$-group is a (\emph{strong order}) \emph{unit} if for each $g\in G$  there is a natural number $n$ such that $n1\geq g$. An Abelian $\ell$-group $G$ equipped with a distinguished unit $1$ is called \emph{unital}, and denoted $(G,1)$. Write $\lA_1$ for the category of unital Abelian $\ell$-groups and of unit-preserving lattice-group homomorphisms.

There is a functor $\Gamma\colon \lA_1\to \MV$ that acts on objects by sending $(G,1)$ to its unit interval $[0,1]\coloneqq\{x\in G\mid 0\leq x \leq 1\}$, and on morphisms by restriction; here, $[0,1]$ is regarded as an MV-algebra under the operations $x\oplus y\coloneqq (x+y)\wedge 1$, $\neg x\coloneqq 1-x$, and $0$. This functor has an adjoint $\Xi\colon\MV\to\lA_1$, and Mundici proved in \cite{Mundici86} that $\Gamma$ and $\Xi$  constitute an equivalence of categories.

The initial object in $\lA_1$ is $(\Z,1)$, and the terminal object is the trivial unital $\ell$-group $(\{0=1\}, 0)$.
In analogy with the relationship between non-unital and unital rings, the category $\lA$ has a zero object and is not coextensive, while the category $\lA_1$ is. Separable unital Abelian $\ell$-groups are defined as for any coextensive category, cf.\ the beginning of Section \ref{s:main}.

An object  $G$ of $\lA$ is \emph{Archimedean} if   whenever $nx\leq y$ holds in $G$ for each positive integer $n$, then  $x\leq 0$; and an object $(G,1)$ of $\lA_1$ is called Archimedean if $G$ is. The following characterisations hold: $(G,1)$ is Archimedean precisely when $\Gamma(G,1)$ is semisimple; and $(G,1)$ is totally ordered and Archimedean precisely when $\Gamma(G,1)$ is simple. H\"older's Theorem for the category $\lA_1$ may be stated as follows: \emph{Any $(G,1)$ that is Archimedean and totally ordered has exactly one morphism to $(\R,1)$, and that morphism is monic} (equivalently, its underlying function is injective). 

Let us say that an object $(G,1)$ of $\lA_1$ is \emph{rational} if it is isomorphic to an ordered subgroup
of the additive group $\Q$ containing $1$, where the order of $G$ is inherited from the natural order of the rationals. Theorem \ref{t:superseparableNew} may be then formulated for the category $\lA_1$ as follows.
\begin{theorem}\label{t:ell-superseparableNew} For any  unital Abelian $\ell$-group $(G,1)$ the following are equivalent.
\begin{enumerate}[\textup{(}i\textup{)}]
\item\label{i:ellnew3} $(G,1)$ is rational.
\item $(G,1)$ is non-trivial, and the unique map $(\Z,1) \rightarrow (G,1)$ is  epic.
\item\label{i:ellnew2} The unique map $(\Z,1) \rightarrow (G,1)$ is monic and epic.
\item $(G,1)$ is totally ordered and Archimedean, and the unique map $(\Z,1) \rightarrow (G,1)$ is  epic.
\end{enumerate}
\end{theorem}

An object $(G,1)$ of $\lA_1$ is \emph{Specker} if its unit-interval MV-algebra $\Gamma(G,1)$ is a Boolean algebra. Write $\Speck_1$ for the full subcategory of $\lA_1$ on the the Specker objects. The inclusion functor $\Speck_1\to\lA_1$ has a right adjoint $\PP\colon \lA_1\to\Speck_1$, the \emph{Pierce functor} for $\lA_1$, and $\PP$ preserves arbitrary coproducts (Theorem \ref{t:Pierce}). Our main result, Theorem \ref{ThmMain}, would be proved for the category $\lA_1$ using this Pierce functor; it can  be phrased as follows.

\begin{theorem}\label{ell-ThmMain} Separable unital Abelian $\ell$-groups coincide with finite products of rational unital Abelian $\ell$-groups.
\end{theorem}
\begin{remark}\label{r:ell-prod}Products in the category $\lA$ are Cartesian products, because $\lA$ is a variety of algebras. On the other hand, while $\lA_1$ is equivalent to a variety by Mundici's cited theorem,  its underlying-set functor is not right adjoint. Indeed, products in $\lA_1$ are not, in general, Cartesian products. However, finite products in $\lA_1$ \emph{are} Cartesian---the product of $(G,1)$ and $(H,1)$ is $(G\times H, (1,1))$ with the Cartesian projections. 
\end{remark}

An Abelian $\ell$-group is called a \emph{simplicial group} if it is isomorphic in $\lA$ to a free Abelian group of finite rank $\Z^r$ equipped with the coordinatewise order. A unit in such a simplicial group is then any element $1\in \Z^r$ whose each coordinate is strictly positive; the pair $(\Z^r,1)$ is called a \emph{unital simplicial group}. These lattice-groups play a key r\^{o}le in the representation theory of dimension groups, see e.g.\ \cite{Goodearl1986}.

An object $(G,1)$ in $\lA_1$ is a unital simplicial group exactly when its unit-interval MV-algebra $\Gamma(G,1)$ is finite.  An object $(G,1)$ is \emph{locally simplicial} if each sublattice subgroup generated by finitely many elements along with $1$ is a unital simplicial group. An object $(G,1)$ in $\lA_1$ is locally simplicial exactly when its unit-interval MV-algebra $\Gamma(G,1)$ is locally finite. Then: \emph{An object $(G,1)$ of $\lA_1$ is separable just  when it is locally simplicial, and $\PP(G,1)$ has finite \textup{(}$\Z$-module\textup{)} rank}\footnote{In the literature on lattice-groups, the condition that $\PP(G,1)$ has finite rank is expressed in  the following traditional manner: the unit of $G$ has finitely many components.} (Corollary \ref{CorCharSepInTermsOfLF}).

Write $\LSimp_1$ for the full subcategory of $\lA_1$ on the locally simplicial objects. \emph{The inclusion functor $\LSimp_1\to\lA_1$ has a right adjoint $\RR\colon \lA_1\to\LSimp_1$} (Corollary \ref{c:rightadjlf}); that is, every $(G,1)$ has an inclusion-largest locally simplicial unital sublattice subgroup. To prove this in the category $\lA_1$ one would introduce the notion of element of `finite-order rank' of a unital Abelian $\ell$-group. It is this notion that motivates the ter\-mi\-no\-logy   we adopted in the context of MV-algebras in Section \ref{s:main}; by way of conclusion of this appendix, we offer a short discussion.

Let $(G,1)$ be a unital Abelian $\ell$-group, let $g\in G$, and let $H$ be the sublattice subgroup of $G$ generated by $g$ and by $1$. If $(H,1)$ is a unital simplicial group $(\Z^r,1)$---equivalently, if the MV-algebra $\Gamma(H,1)$ is finite---then we call $g$ an element of \emph{finite order-rank $r$}. This notion of rank crucially depends on the interplay between the lattice and the group structure, and is not reducible to the linear notion of rank. To explain why, let us preliminarly observe  that a simplicial group $\Z^r$ enjoys the finiteness property that its positive cone $(\Z^r)^+$---that is, the monoid of non-negative elements of $\Z^r$---is finitely generated as a monoid. Next, let us point out  that  the underlying  group of the Abelian $\ell$-group $H$ generated by $g$ and $1$ in $G$ is necessarily free: indeed,  any finitely generated object of $\lA$ has free underlying group, as was proved in \cite{GlassMarra}. The $\Z$-module rank of $H$ is at most countably infinite, because $H$ is countable. But even if we assume the rank of $H$ is finite, the unit-interval $\Gamma(H,1)$ may be infinite, and in that case  the lattice order of $\Z^r\cong H$ cannot be simplicial---and indeed, one can prove that the monoid $H^+$ cannot be finitely generated. Hence, the condition that  the sublattice subgroup $H$ of $G$ generated by $g$ and $1$ is simplicial is strictly stronger than the condition that $H$ has finite $\Z$-module rank. To illustrate, consider the subgroup $H$ of $\R$ generated by an irrational number $\rho\in \R$ together with $1$; then $H\cong \Z^2$ as groups, the total order inherited by $\Z^2$ from $\R$ is palpably not simplicial, the positive cone $H^+$ can be shown not to be  finitely generated by an easy direct argument, and $\Gamma(H,1)$ is an infinite simple MV-algebra.

\bibliography{biblio}
\bibliographystyle{amsplain}

\end{document}